\documentclass[letterpaper, reqno,10 pt]{amsart}
\usepackage[margin=1.0in]{geometry}
\usepackage{amsmath,amssymb,amsthm}
\usepackage{mathrsfs} 
\usepackage{etoolbox}
\usepackage{comment}
\usepackage{graphicx}
\usepackage{wrapfig}
\usepackage{enumerate}

\makeatletter
\patchcmd{\subsubsection}{\scshape}{\bf}{}{}
\patchcmd{\subsubsubsection}{\scshape}{\bf}{}{}
\makeatother

\makeatletter
\renewcommand{\tocsection}[3]{%
  \indentlabel{\@ifnotempty{#2}{\bfseries\ignorespaces#1 #2\quad}}\bfseries#3}
\renewcommand{\tocsubsection}[3]{%
  \indentlabel{\@ifnotempty{#2}{\ignorespaces#1 #2\quad}}#3}
\renewcommand{\tocsubsubsection}[3]{%
  \indentlabel{\@ifnotempty{#2}{\hspace{1.4cm}\ignorespaces#1 #2\quad}}#3}

\makeatletter
\pretocmd{\chapter}{\addtocontents{toc}{\protect\addvspace{15\p@}}}{}{}
\pretocmd{\section}{\addtocontents{toc}{\protect\addvspace{5\p@}}}{}{}

\newcommand\@dotsep{4.5}
\def\@tocline#1#2#3#4#5#6#7{\relax
  \ifnum #1>\c@tocdepth 
  \else
    \par \addpenalty\@secpenalty\addvspace{#2}%
    \begingroup \hyphenpenalty\@M
    \@ifempty{#4}{%
      \@tempdima\csname r@tocindent\number#1\endcsname\relax
    }{%
      \@tempdima#4\relax
    }%
    \parindent\z@ \leftskip#3\relax \advance\leftskip\@tempdima\relax
    \rightskip\@pnumwidth plus1em \parfillskip-\@pnumwidth
    #5\leavevmode\hskip-\@tempdima{#6}\nobreak
    \leaders\hbox{$\m@th\mkern \@dotsep mu\hbox{.}\mkern \@dotsep mu$}\hfill
    \nobreak
    \hbox to\@pnumwidth{\@tocpagenum{\ifnum#1=1\bfseries\fi#7}}\par
    \nobreak
    \endgroup
  \fi}
\AtBeginDocument{%
\expandafter\renewcommand\csname r@tocindent0\endcsname{0pt}
}
\def\l@subsection{\@tocline{2}{0pt}{2.5pc}{5pc}{}}
\makeatother

\usepackage{esint}
\usepackage{fancybox}
\usepackage{color}   
\usepackage{hyperref}
\hypersetup{
    colorlinks=true, 
    linktoc=all,     
    linkcolor=blue,  
}

\newcommand{\Addresses}{{
  \bigskip
  \footnotesize

	\medskip
	\textsc{Instituto de Ciencias Matem\'aticas (CSIC), C. Nicol\'as Cabrera, 13-15, 28049 Madrid, Spain}\par\nopagebreak
  \textit{E-mail address}:  \texttt{daniel.lear@icmat.es}  }}

\newtheorem{thm}{Theorem}[section]
\newtheorem{cor}[thm]{Corollary}

\newtheorem*{prop*}{Proposition}
\newtheorem{defi}[thm]{Definition}

\newtheorem{lemma}[thm]{Lemma}

\newcommand{\R}{\mathbb{R}}
\newcommand{\Q}{\Bbb{Q}}

\newcommand{\Z}{\Bbb{Z}}

\newcommand{\T}{\mathbb{T}}

\begin{document}
\vspace*{-0.1 cm}
\title[ON THE  NON-DIFFUSIVE MAGNETO-GEOSTROPHIC EQUATION]{\textsc{ON THE  NON-DIFFUSIVE MAGNETO-GEOSTROPHIC EQUATION}}
\author{Daniel Lear}
\date{\today}
\maketitle
\vspace{-1 cm}
\begin{abstract}
Motivated by an equation arising in magnetohydrodynamics, we address the well-posedness theroy for the non-diffusive magneto-geostrophic equation. Namely, an active scalar equation in which the divergence-free drift velocity is one derivative more singular that the active scalar. In \cite{Friedlander-Vicol_3}, the authors prove that the non-diffusive equation is ill-posed in the sense of Hadamard in Sobolev spaces, but locally  well posed in spaces of analytic functions. 
Here, we give an example of a steady state that is nonlinearly stable for periodic perturbations with initial data localized in frequency straight lines crossing the origin. For  such well-prepared data, the local existence and uniqueness of solutions can be obtained in Sobolev spaces and the global existence holds under a size condition over the $H^{5/2^{+}}(\T^3)$ norm of the perturbation.
\end{abstract}


\tableofcontents

\section{Introduction}\label{Sec_1}
The geodynamo is the process by which the Earth's magnetic field is created and sustained by the motion of the fluid core, which is composed of a rapidly rotating, density stratified, electrically conducting fluid.  The convective processes in the core that produce the velocity fields required for dynamo action are a combination of thermal and compositional convection. The full dynamo problem requires the examination of the full 3D partial differential equations governing convective, incompressible \textit{magnetohydrodyamics} (MHD).

It is therefore reasonable to attempt to gain some insight into the geodynamo by considering a reduction of the full MHD equations to a system that is more tractable, but one that retains many of the essential features relevant to the physics of the Earth's core.

Recently, Moffatt and Loper  \cite{Moffatt}, \cite{Moffatt-Loper} proposed the \textit{magneto-geostrophic} equation (MG) as a model for the geodynamo which is a reduction of the full MHD system.  The physical postulates of this model are the following: slow cooling of the Earth leads to slow solidification of the liquid metal core onto the solid inner core and releases latent heat of solidification that drives compositional convection in the fluid core.

\subsection{Governing equations:}
We first present the full coupled three-dimensional MHD equations for the evolution of the velocity vector $\mathbf{U}(\mathbf{x},t)$, the magnetic field vector $\mathbf{B}(\mathbf{x},t)$ and the buoyancy field $\Theta(\mathbf{x},t)$ in the Boussinesq approximation and written in the frame of reference rotating with angular velocity $\Omega$. For simplicity, we have assumed that the axis of rotation and the gravity $g$ are aligned in the direction of $e_3$.

Following the notation of Moffatt and Loper \cite{Moffatt-Loper} we obtain the dimensionless equations
\begin{align}\label{mhd}
N^2[R_0(\partial_t \mathbf{U}+\mathbf{U}\cdot\nabla \mathbf{U})+e_3\times \mathbf{U}]&=-\nabla P+\left(e_2\cdot\nabla\right) \mathbf{b}+R_m \mathbf{b}\cdot\nabla \mathbf{b}+N^2\Theta e_3+\epsilon_\nu \Delta \mathbf{U},\nonumber\\
R_m[\partial_t \mathbf{b}+\mathbf{U}\cdot\nabla \mathbf{b}-\mathbf{b}\cdot\nabla \mathbf{U}]&=\left(e_2\cdot\nabla\right) \mathbf{U}+\Delta \mathbf{b},\nonumber\\
\partial_t\Theta+\mathbf{U}\cdot\nabla\Theta&=\epsilon_\kappa\Delta\Theta,\nonumber\\
\nabla\cdot \mathbf{U}&=0,\nonumber\\
\nabla\cdot \mathbf{b}&=0,
\end{align}
where $P$  is the sum of the fluid and magnetic pressures, $\epsilon_{\nu}$ is the (non-dimensional) kinematic viscosity and $\epsilon_{\kappa}$ is the (non-dimensional) thermal diffusivity. Here $(e_1,e_2,e_3)$ denote the  Cartesian unit vectors.

Following \cite{Moffatt}, we have assumed in (\ref{mhd}) that the magnetic field in the core is of the form
$$\textbf{B}(\textbf{x},t)=\textbf{B}_0+\textbf{b}(\textbf{x},t),$$
where $\textbf{B}_0$ results from dynamo action and can be considered as locally uniform and steady, and a perturbation field $\textbf{b}(\mathbf{x},t)$ induced by the flow $\textbf{U}(\textbf{x},t)$ across $\textbf{B}_0$. Our choice of $\textbf{B}_0\equiv \beta\, e_2$ as the underlying magnetic field is consistent with the models where the magnetic field is believed to be predominantly toroidal due  to the strong influence of differential rotation \cite{Moffatt-Loper}.\vspace{0.2 cm}

The dimensionless parameters in (\ref{mhd}) are the followings:
\begin{align*}
N^2&=2 \Omega \mu_0 \eta \rho/\beta^2 \qquad &\text{inverse of the Elsasser number},\\
R_o&=V/2L\Omega \qquad &\text{Rossby number},\\
R_m&=VL/\eta \qquad & \text{magnetic Reynolds number},\\
\epsilon_{\nu}&=\nu \eta \mu_0 \rho /\beta^2 L^2 \qquad &\text{inverse square of the Hartman number},\\
\epsilon_{\kappa}&=\kappa/LV \qquad &\text{inverse of the Peclet number}.
\end{align*}
Here $\nu$ is the kinematic viscosity, $\eta$ is the magnetic diffusivity of the fluid, $\kappa$ is the molecular diffusivity of the compositional variation that creates an ambient density $\rho$ and $\mu_0=4\pi \times 10^{-7} \text{NA}^{-2}.$
We adopt as velocity scale $V=\Theta_0 g/2\Omega$ where $\Theta_0$ is the typical amplitude of $\Theta$, and that the length-scale of these variations is $L$.\vspace{0.2 cm}

The orders of magnitude of the nondimensional parameters are motivated by the physical postulates of the Moffatt and Loper model. For the regions in the Earth's fluid core modeled in (\ref{mhd}), it is argued in \cite{Moffatt-Loper} that the parameters have the following orders of magnitude:
$$N^2 \approx 1, \qquad R_o\approx 10^{-3}, \qquad R_m\approx 1, \qquad \epsilon_{\nu}\approx 10^{-8}, \qquad \epsilon_{\kappa}\approx 10^{-8}.$$
The values of $\nu$ and $\kappa$ are speculative, but likely to be extremely small. For a detailed discussion of plausible ranges of the physical parameters that are appropiate for the geodynamo, we refer the reader to  \cite{Ghil-Childress}.

According to Moffat and Loper, the magnetic Reynolds number is relatively small. Then, their model neglects the terms multiplied by $R_o$ and $R_m$ in comparison with the remaining terms. However, we will for the moment retain the viscous and diffusive terms since it involve the highest derivatives.

For the reasons given above, we now drop in (\ref{mhd}) the terms involving the Rossby number $R_o$ and the magnetic Reynolds number $R_m$. Then, we obtain the following reduced system:
\begin{align}\label{mhd_simplified}
N^2 [e_3\times \mathbf{U}]&=-\nabla P+\left(e_2\cdot\nabla\right) \mathbf{b}+N^2\Theta e_3+\epsilon_\nu \Delta \mathbf{U},\nonumber\\
0&=\left(e_2\cdot\nabla\right) \mathbf{U}+\Delta \mathbf{b},\nonumber\\
\partial_t\Theta+\mathbf{U}\cdot\nabla\Theta&=\epsilon_\kappa\Delta\Theta,\nonumber\\
\nabla\cdot \mathbf{U}&=0,\nonumber\\
\nabla\cdot \mathbf{b}&=0.
\end{align}
Essentially this means that the evolution equations for the coupled velocity $\textbf{U}$ and magnetic field \textbf{b} take a simplified ``quasi-static'' form. This system encodes the vestiges of the physics in the problem, namely the Coriolis force, the Lorentz force and gravity.\vspace{0.2 cm}

The behavior of the model is dramatically different when the parameters $\epsilon_{\nu}$ and $\epsilon_{\kappa}$ are present or absent. Since both parameters multiply a Laplacian term, their presence is smoothing. In the present paper we focus our attention in the inviscid case $(\epsilon_{\nu}=0)$. The mathematical properties of the model under the presence of viscosity  have been addressed in a recent sequence of different articles \cite{Friedlander-Suen_3}, \cite{Friedlander-Suen_1} and \cite{Friedlander-Suen_2}.\vspace{0.2 cm}

\subsection{The MG equation:}
A linear relationship can be established between the divergence-free vector fields $\mathbf{U}$ and $\mathbf{b}$ and the scalar \nolinebreak$\Theta$, wherein $\Theta$ will now be regarded as known, thanks to the reduced system:
\begin{align}\label{mhd_linear}
N^2 [e_3\times \mathbf{U}]&=-\nabla P+\left(e_2\cdot\nabla\right) \mathbf{b}+N^2\Theta e_3,\nonumber\\
0&=\left(e_2\cdot\nabla\right) \mathbf{U}+\Delta \mathbf{b},
\end{align}
along with the incompressibility condition
$$\nabla\cdot \mathbf{U}=0, \qquad  \nabla\cdot \mathbf{b}=0.$$

We note that the ratio of the Coriolis to Lorentz forces in their model is of order 1, so for notational simplicity we have set this parameter, denoted by $N^2$ equal to 1.
Following \cite[p.~297]{Friedlander-Vicol_1}, manipulations of the linear system (\ref{mhd_linear}) gives, in component form
\begin{equation}\label{mhd_componentes}
\left\{
\begin{array}{rl}
U_1&=-D^{-1}\left(\partial_2 P+\Gamma\partial_1 P \right),\\
U_2&=\phantom{-}D^{-1}\left(\partial_1 P-\Gamma\partial_2 P \right),\\
\partial_3 U_3 &=\phantom{-}D^{-1}\Gamma\Delta_{H}P,\\
\partial_3 \Theta&=\phantom{-} \left(\Gamma^2\Delta_H D^{-1}+\partial_{3}^2\right)P,
\end{array}
\right.
\end{equation}
where the operators $\Gamma, D$ and $\Delta_{H}$ are defined as
$$\Gamma:=-(-\Delta)^{-1}\partial_2^2, \qquad D:=1+\Gamma^2, \qquad \Delta_{H}:=\partial_1^2+\partial_2^2.$$
Although the physically relevant boundary for a model of the Earth's fluid core is a spherical annulus, for mathematical tractability we considered the system on the domain $\mathbf{x}\in\mathbb{T}^2\times\R$. This can be seen as a first step before considering the case $\T^2\times \text{I}$ with appropiate boundary conditions in the vertical variable.

In order to uniquely determine $U_3$ and $\Theta$ form (\ref{mhd_componentes}), we restrict the system to the function spaces of zero vertical mean, i.e. $\int_{\R}U_3\,dx_3=\int_{\R}\Theta\,dx_3=0$. In fact, without such a restriction the system is not well \nolinebreak defined.\vspace{0.2 cm}

We can integrate the last equation of (\ref{mhd_componentes}) and use the zero vertical mean assuption to obtain that
$$\Theta=A[P],$$
where the operator $A$ is formally defined as $A:=\partial_3^{-1} \left(\Gamma^2\Delta_H D^{-1}+\partial_{3}^2\right)$ in the physical space. On one hand, we now use (\ref{mhd_componentes}) to represent $U_1, U_2$ and $U_3$ in terms of $\Theta$:
\begin{align}\label{U_M[theta]}
U_1&=-D^{-1}\left(\partial_2 +\Gamma\partial_1  \right)\left(A^{-1}[\Theta]\right)\equiv M_1[\Theta],\nonumber\\
U_2&=\phantom{-}D^{-1}\left(\partial_1 -\Gamma\partial_2  \right)\left(A^{-1}[\Theta]\right)\equiv M_2[\Theta],\nonumber\\
U_3 &=\phantom{-}D^{-1}\Gamma\Delta_{H}\left(D^{-1}\Gamma\Delta_{H}+\partial_3^2\right)^{-1}[\Theta]\equiv M_3[\Theta].
\end{align}
\textbf{Remark:} A precise expression of the operator $\textbf{M}$ will be given as a Fourier multiplier operator in  Section \nolinebreak\ref{Sec_2}.\vspace{0.2 cm}
On the other hand, the magnetic vector field $\textbf{b}$ is computed from the scalar $\Theta$ thanks to (\ref{mhd_linear}) via the operator
$$\textbf{b}_j=(-\Delta)^{-1}\partial_2 M_j[\Theta], \qquad \text{for}\quad  j\in\{1,2,3\}.$$

The sole remaining nonlinearity in the system comes from the coupling of (\ref{mhd_linear}) and the evolution equation for the scalar bouyancy $\Theta$. The active scalar equation for $\Theta$ that contains the non-linear process in Moffatt's model is precisely:
\begin{equation}\label{active_scalar}
\left\{
\begin{array}{rl}
\partial_t \Theta +\mathbf{U}\cdot\nabla\Theta&=\epsilon_\kappa \,\Delta\Theta,\\
\nabla\cdot\mathbf{U}&=0,
\end{array}
\right.
\end{equation}
where the divergence-free velocity $\mathbf{U}$ is explicitly obtained from the bouyancy as $\mathbf{U}=\mathbf{M}[\Theta]$ where $\mathbf{M}$ is the non-local differential operator of order 1 defined in (\ref{U_M[theta]}). We describe below the precise form of that operator.\vspace{0.2 cm}

\noindent
\textbf{Remark:} As we said, we consider for simplicity the domain $\T^2\times\R$. Note that, without loss of generality we may assume that $\int_{\T^2\times\mathbb{R}}\Theta(\mathbf{x},t)\,d\mathbf{x}=0$ for all $t\geq 0$, since the mean of $\Theta$ is conserved by the \nolinebreak flow.\vspace{0.2 cm}

In the following, we refer to the evolution equation (\ref{active_scalar}) with singular drift velocity $\mathbf{U}$ given by (\ref{U_M[theta]}) as the \textit{magneto-geostrophic} equation (MG). In addition, we will distinguish between \textit{diffusive} ($\epsilon_\kappa>0$) and \textit{non-diffusive} case ($\epsilon_\kappa=0$). In the Earth's fluid core the value of the diffusivity $\epsilon_\kappa$ is very small. Hence it is relevant to address both the diffusive evolution, and the non-diffusive version where $\epsilon_\kappa=0.$\vspace{0.2 cm}

The aim of the present paper is to show that the  Cauchy problem for the \textit{non-diffusive} MG equation is  well-posed with respect to some periodic perturbations around a specific steady profile, in the topology of a certain Sobolev space. In the next section, we state the main result of this paper at a descriptive level.

\subsection{Diffusive vs. non-diffusive MG equation:}
In order to study this dichotomy, we recall the following: In the theory of differential equations, it is classical to call a Cauchy problem \textit{well-posed}, in the sense of  Hadamard,  if  given  any  initial  data  in  a  functional  space $X$,  the  problem  has  a  unique  solution  in $L^{\infty}(0,T;X)$, with $T$ depending only on the $X$-norm of the initial data, and moreover the solution map $Y\mapsto L^{\infty}(0,T;X)$ satisfies strong continuity properties, e.g. it is uniformly continuous, Lipschitz, or even $\mathcal{C}^{\infty}$ smooth,  for a sufficiently nice space $Y\subset X$. If one of these properties fail,  the Cauchy problem is called \textit{ill-posed}.\vspace{0.2 cm}

Considering this, both systems have contrasting properties:
\begin{itemize}
	\item Diffusive MG equation: For $\epsilon_\kappa>0$ the equation is globally well-posed and the solutions are  $\mathcal{C}^{\infty}$ smooth for positive times, as it is proved in the papers \cite{Friedlander-Vicol_1} and \cite{Friedlander-Vicol_2}.
	\item Non-diffusive MG equation: For $\epsilon_\kappa=0$, in \cite{Friedlander-Vicol_3} the authors prove that the equation is
ill-posed in the sense of Hadamard in Sobolev spaces, but locally well-posed in spaces of analytic functions.
\end{itemize}
More specifically, we mention that for analytic initial data, the \textit{non-diffusive} MG equation is indeed locally well-posed in the class of real-analytic functions in the spirit of a Cauchy-Kowalewskaya result, since each term in the equation loses at most one derivative.

Moreover, in the same article the authors prove  that the solution map associated to the Cauchy problem is not Lipschitz continuous with respect to perturbations in the initial data around a specific steady profile $\Theta_0(x_3):=\sin(m \,x_3)$ for some integer $m\geq 1$, in the topology of a certain Sobolev space.

The proof consists of a linear and a nonlinear step. After linearizing the problem around $\Theta_0$, the authors 
employ techniques from continued fractions in order to construct an
unstable eigenvalue for the linearized operator.  Once these eigenvalues are exhibited, one may use a fairly robust argument to show that this severe linear ill-posedness implies the Lipschitz ill-posedness for the nonlinear problem.

The use of continued fractions in a fluid stability problem
was introduced in \cite{Sinai} for the Navier-Stokes equations and later adapted for the Euler
equations in \cite{Friedlander-Strauss-Vishik}.\vspace{0.2 cm}

Hence, without the Laplacian to control the unbounded operator $\textbf{M}$ the situation is dramatically different  from the $\textit{diffusive}$ case $\epsilon_{\kappa}>0$. For the above, the problem of the \textit{fractionally diffusive} MG equation arise naturally. This is, one can replace the Laplacian  by nonlocal operators, such as $-(-\Delta)^{\gamma}$ for $\gamma\in (0,1)$. This situation, which is non-physical but mathematically interesting, it was addressed in \cite{Friedlander-Rusin-Vicol_fractional}.  In the subcritical range $\gamma\in(1/2,1)$ the equation is locally well-posed, while it is Hadamard Lipschitz ill-posed for $\gamma\in(0,1/2)$.  At the critical value $\gamma=1/2$ the problem is globally well-posed for suitably small initial data, but is ill-posed for sufficiently large initial data. 

A further feature of interest is that the anisotropy of the
symbol $\textbf{M}$ can be explored as in \cite{Friedlander-Rusin-Vicol_fractional} to obtain an improvement in the regularity of the solutions when the initial data
is supported on a plane in the Fourier space.  For such well-prepared initial data the local existence and uniqueness of solutions can be obtained for all values $\gamma\in(0,1)$, and the global existence holds for all initial
data when $\gamma\in(1/2,1)$. \vspace{0.2 cm}

\subsection{Singular active scalar}
One may view the MG equation as an example of a \textit{singular} active scalar since the drift velocity is given in terms of the advected scalar by a constitutive law which is losing derivatives.\vspace{0.2 cm}

Active scalars appear in many problems coming from fluid mechanics. It consists of solving the Cauchy problem for the transport equation:
\begin{equation}\label{hierarchy_active_scalar}
\left\{
\begin{array}{rl}
\partial_t \Theta +\mathbf{U}\cdot\nabla\Theta&=-\epsilon_\kappa \,(-\Delta)^{\gamma}\Theta,\\
\nabla\cdot\mathbf{U}&=0,
\end{array}
\right.
\end{equation}
where the vector field $\mathbf{U}$ is related to $\Theta$ by some operator. We remark that the MG equation fall into a hierarchy of active scalar equations arising in fluid dynamics in terms of the nature of the operator that produces the drift velocity from the scalar field:\vspace{0.2 cm}

\begin{center}
\textsc{Hierarchy of active scalar equations:}
\end{center}
\begin{enumerate}[i)]
	\item Inviscid MG equation $(\varepsilon_{\nu}=0)$: \hspace{3 cm} $\mathbf{U}=\mathbf{M}[\Theta]$ \hfill \textsf{Singular order 1}
		\begin{itemize}
			\item $\epsilon_\kappa=0$: Hadamard ill-posed.
			\item $\epsilon_\kappa>0$: Critical case, globally well-posed.
		\end{itemize}
		
	\item SGQ equation (see \cite{Constantin-Majda-Tabak} and  \cite{Caffarelli-Vasseur}, \cite{Constantin-Vicol}, \cite{Kiselev-Nazarov-Volberg}): \hspace{1.7 cm} $\mathbf{U}=\nabla^{\perp}(-\Delta)^{-1/2}\Theta$ \hfill \textsf{Singular order 0}
		\begin{itemize}
			\item $\epsilon_\kappa=0$: Open.
			\item $\epsilon_\kappa>0$: Critical case, globally well-posed.
		\end{itemize}
		
	\item[iii)] Burgers equation (see \cite{Kiselev-Nazarov-Shterenberg}):	\hspace{3.6 cm} $\mathbf{U}=\Theta$ \hfill \textsf{Order 0}
	\begin{itemize}
			\item $\epsilon_\kappa=0$: Blow-up.
			\item $\epsilon_\kappa>0$: Critical case, globally well-posed.
		\end{itemize}
		
	\item[iv)] 2D Euler equation in vorticity form: \hspace{2.25 cm} $\mathbf{U}=\nabla^{\perp}(-\Delta)^{-1}\Theta$ \hfill \textsf{Smoothing degree 1}\\
	Globally well-posed.
	
\end{enumerate}

We emphasize that the mechanism producing ill-posedness is not merely the order one derivative loss in the map $\Theta \mapsto \mathbf{U}$. Rather, it is the combination of the derivative loss with the anisotropy of the symbol $\textbf{M}$ and the fact that this symbol is even. We note that the even nature of the symbol of $\textbf{M}$ plays a central role in the proof of non-uniqueness for $L^{\infty}$-weak solutions to the \textit{non-diffusive} MG equation proved in \cite{Shvydkoy}, via methods from convex integration. In contrast, an example of an active scalar equation where the map $\Theta \mapsto \mathbf{U}$ is unbounded, but given by an odd Fourier multiplier, is the generalized SQG equation where $\mathbf{U}=\nabla^{\perp}(-\Delta)^{-\frac{1-\gamma}{2}}\Theta$ and $0<\gamma\leq 1$. This equation was recently shown in \cite{Chae-Constantin-Cordoba-Gancedo-Wu} to give a locally well-posed problem in Sobolev spaces.

\subsubsection{On the lack of well-posedness for the MG equation in Sobolev spaces} Let us briefly discuss why the evenness of the operator $\textbf{M}$ breaks the classical proof of local existence in Sobolev spaces for the ($\epsilon_\kappa=0$) \textit{non-diffusive}  MG equation. To see why one may not use the standard energy-approach to obtain local well posedness, we point out that in the energy estimate for (\ref{hierarchy_active_scalar}) there are only two terms  which seem to prevent closing the estimate at the $H^s$ level: 
$$T_{bad}=\int \Lambda^s\Theta\,\Lambda^s\mathbf{U}_j\,\partial_j\Theta=\int \Lambda^s\Theta\,\Lambda^s M_j[\Theta]\,\partial_j\Theta  \qquad \text{and} \qquad T_{good}=\int \Lambda^s\Theta\,\mathbf{U}_j\,\partial_j\Lambda^s\Theta,$$
where we denoted $\Lambda := (-\Delta)^{-1/2}$. Since $\nabla\cdot\mathbf{U}=0$, upon integrating by parts we have $T_{good}=0$. On the other hand, the term $T_{bad}$ does not vanish in general. The only hope to treat the term $T_{bad}$ would be to discover a commutator structure. However, since $\textbf{M}$ is not anti-symmetric, i.e. even in Fourier space, we cannot write $T_{bad}=-\mathcal{T}$, where
$$\mathcal{T}=\int M_j\left[\Lambda^{s}\Theta\,\partial_j\Theta \right] \,\Lambda^{s}\Theta=T_{bad}+\int\left[M_j,\partial_j\Theta\right]\Lambda^{s}\Theta\,\Lambda^{s}\Theta=T_{bad}+\mathcal{S}.$$
If you could do this, a suitable commutator estimate of Coifman-Meyer type would close the estimates at the level of Sobolev spaces. Instead we have that $T_{bad}=\mathcal{T}$. This is the main reason why we are unable to close  estimates at the Sobolev level.

\subsection{Preliminares}
This section contains a few auxiliary results used in the paper. In particular, we recall the, by now classical, product and commutator estimates, as well as the Sobolev embedding inequalities. Proofs of these results can be found for instance in \cite{Kenig-Ponce-Vega},\cite{Stein} and \cite{Tao}. 

\begin{lemma}[{\bf Product estimate}]\label{product_estimate}
\label{prop:calculus}
If $s > 0$, then for all $f,g \in H^s\cap L^\infty$ we have the estimate
	\begin{align}
		&\|\Lambda^s(fg)\|_{L^2} \lesssim \left( \|f\|_{L^{\infty}} \|\Lambda^s g\|_{L^{2}} + \|\Lambda^sf\|_{L^{2}} \|g\|_{L^{\infty}}\right). \label{eq:prop:product}
	\end{align}
\end{lemma}

In the case of a commutator we have the following estimate.

\begin{lemma}[{\bf Commutator estimate}]\label{commutator_estimate}
	Suppose that $s >0$. Then for all  $f,g \in \mathcal{S}$ we have the estimate
	\begin{align}
		&\|\Lambda^s(fg) - f \Lambda^{s} g \|_{L^2} \lesssim \left(\| \nabla f\|_{L^{\infty}}\|\Lambda^{s-1} g\|_{L^{2}} + \|\Lambda^s f\|_{L^{p}}\|g\|_{L^{p'}} \right) \label{eq:prop:commutator}
	\end{align}
	where $\tfrac{1}{2}=\tfrac{1}{p}+\tfrac{1}{p'}$ and $p\in (1,\infty)$.
\end{lemma}

Moreover, the following Sobolev embeddings holds:
\begin{itemize}
	\item $W^{s,p}(\T^d)\subset L^q(\T^d)$ continuosly if $s<d/p$ and $p\leq q \leq dp/(d-sp)$.
		
	\item $W^{s,p}(\T^d)\subset C^k(\overline{\T^d})$ continuosly if $s>k+d/p$.
\end{itemize}
\vspace{0.2 cm}

\subsection{Notation \& Organization:}
To avoid clutter in computations, function arguments (time and space) will be omitted whenever they are obvious from context. Finally, we use the notation $f\lesssim g$ when there exists a constant $C>0$ independent of the parameters of interest such that $f\leq Cg$.\vspace{0.2 cm}

\noindent
In Section \ref{Sec_2} we begin by setting up the perturbated problem around a specific  steady state. Then, we state a less technical version of the main theorem and we give some of the ideas behind the proof. Here, we collect some useful technical lemmas about the behaviour of $\widehat{\textbf{M}}$ on suitable subsets of the frequency domain. In Section \ref{Sec_3} we embark on the proof of a local existence result for frequency-localized initial data following the ideas of \cite{Friedlander-Rusin-Vicol_fractional}. The core of the article is the proof of the main theorem in Section \ref{Sec_4}. We start by the \textit{a priori} energy estimates given in Section \ref{Sec_4.1}. This is followed by an explanation of the decay given by the linear semigroup in Section \ref{Sec_4.2}. Finally, in Section \ref{Sec_4.3} we exploit a bootstrapping argument to prove our theorem.\vspace{0.2 cm}

\section{The perturbated system}\label{Sec_2}

With all the above in mind, we seek to find a steady state around which the \textit{non-diffusive} MG equation is well-posed. Following the same idea used initially in \cite{Friedlander-Rusin-Vicol_fractional}, we take advantage of the anisotropy of the symbols $\textbf{M}=(M_1,M_2,M_3)$ given by (\ref{U_M[theta]}) and observe an interesting phenomenon: when the initial perturbation  is localized in the frequency space, it is possible to prove a well-posedness result for the ensuing solution.

If the frequency of the initial perturbation of the steady state lies on a suitable region of the Fourier space, then the operator $\textbf{M}$ behaves like an order zero opeartor, and hence the corresponding velocity is as smooth as the advected scalar. This enables us to obtain a well-posedness result over the generic setting when no conditions of the Fourier spectrum of the initial perturbation are imposed.

\subsection{The steady states} When studying fluid equations, it is often helpful to have a good understanding of the exact steady states of the system. The kinds of exact solutions we are interested are the simplest possible steady state, namely $\mathbf{U}=0$ and $\Theta_0=\Omega(x_3)$ for some function $\Omega$ with $\int_{\mathbb{R}}\Omega(x_3)\, dx_3=0$.\vspace{0.2 cm}

The basic problem is to consider $\Theta_0$ a given equilibrium state  and to study the dynamics of solutions which are close to it in a suitable sense. Now, we write the scalar and the velocity as
\begin{align*}
\Theta(\mathbf{x},t)&=\Omega(x_3)+\theta(\mathbf{x},t),\\
U(\mathbf{x},t)&=\phantom{\Omega(x_3)+ } \,\, \mathbf{u}(\mathbf{x},t),
\end{align*}
and the pressure term is written in a more convenient way as
$$P(\mathbf{x},t)=\Omega(0)+\int_{0}^{x_3}\Omega(s)\,ds+p(\mathbf{x},t).$$
Then, putting this ansatz in (\ref{mhd_componentes}) we obtain
\begin{equation}\label{mhd_componentes_new}
\left\{
\begin{array}{rl}
u_1&=-D^{-1}\left(\partial_2 p+\Gamma\partial_1 p \right),\\
u_2&=\phantom{-}D^{-1}\left(\partial_1 p-\Gamma\partial_2 p \right),\\
\partial_3 u_3 &=\phantom{-}D^{-1}\Gamma\Delta_{H}p,\\
\partial_3 \theta&=\phantom{-} \left(\Gamma^2\Delta_H D^{-1}+\partial_{3}^2\right)p.
\end{array}
\right.
\end{equation}

As before, in order to uniquely determine $u_3$ and $\theta$ from (\ref{mhd_componentes_new}), we restrict the system to the function spaces where $u_3$ and $\theta$ have zero vertical mean. Hence, we can integrate the last equation of the previous system and use the zero vertical mean assuption to obtain that
$$\theta(\mathbf{x},t)=A[p](\mathbf{x},t).$$
\noindent
\textbf{Remark:} If we impose that  $\theta(\mathbf{x},t)$ and $p(\mathbf{x},t)$ are periodic functions in the three variables $\mathbf{x}=(x_1,x_2,x_3)$, then $A$ is invertible on the space of functions with zero  $x_3$-mean and has an expression as a Fourier multiplier.\vspace{0.2 cm}

For periodic perturbations in the three variables, the operator $A$ is a Fourier multiplier with symbol
\begin{equation*}
\hat{A}(\mathbf{k})=\frac{1}{i k_3}\,\frac{k_3^2 |\mathbf{k}|+k_2^4}{|\mathbf{k}|^4+k_2^4 },
\end{equation*}
where the Fourier variable $\mathbf{k}\in\Z^3_{\star}:=\Z^3\setminus\{k_3=0\}$, by our vertical mean-free assumption.
After that, we can use (\ref{mhd_componentes_new}) and (\ref{mhd_linear}) to represent $\mathbf{u}$ and $\mathbf{b}$ in terms of $\theta$ via
\begin{equation}\label{ref_velocity}
u_j=M_j[\theta] \qquad \text{and} \qquad \textbf{b}_j=(-\Delta)^{-1}\partial_2 M_j[\theta] \qquad \text{for} \quad j\in\left\lbrace 1,2,3\right\rbrace.
\end{equation}
Note that the operators $\left\lbrace M_j \right\rbrace_{j=1}^{3}$ are Fourier multipliers with symbols given explicity for $\mathbf{k}\in\Z^3_{\star}$ by
\begin{equation*}
\widehat{M}_1(\mathbf{k}) := \frac{k_2 k_3 |\mathbf{k}|^2 -  k_1 k_2^{2} k_3}{ k_3^{2} |\mathbf{k}|^2 +  k_2^{4}}, \qquad \widehat{M}_2(\mathbf{k}) := \frac{-k_1 k_3 |\mathbf{k}|^2 -  k_2^{3}k_3}{ k_3^{2} |\mathbf{k}|^2 +  k_2^{4}},
\qquad \widehat{M}_3(\mathbf{k}) := \frac{ k_2^{2}(k_1^{2} + k_2^{2})}{ k_3^{2} |\mathbf{k}|^2 +  k_2^{4}}.
\end{equation*}
On $\{k_3=0\}$ we let $\widehat{M}_j(\mathbf{k})=0$, since for consistency of the model we have that $\theta$ and $U_3$ have zero $x_3$-mean. It can be directly checked that $k_{j} \cdot \widehat{M}_j (\mathbf{k}) = 0$ and hence the velocity field $\mathbf{u}$ given by (\ref{ref_velocity}) is divergence-free.\vspace{0.2 cm}

\subsubsection{The Fourier multiplier operator}
We study  the properties and behavior of the Fourier mulplier operator $\textbf{M}$ obtained from (\ref{ref_velocity}), which relates $\mathbf{u}$ and $\theta$. It is important to note that although the symbols $\widehat{M}_{j}$ are zero-order homogenous under the isotropic scaling $\mathbf{k} \to \lambda \mathbf{k}$, due to their anisotropy the symbols $\widehat{M}_j$ are {not} bounded functions of $\mathbf{k}$. In fact, it may be shown that $|\widehat{\textbf{M}}(\mathbf{k})| \lesssim  |\mathbf{k}|$ and this bound is sharp. To see this, note that whereas in the region of Fourier space where $|k_{1}| \leq \max \{ |k_{2}|, |k_{3}|\}$ the $\widehat{M}_j$ are bounded by a constant, uniformly in $|\mathbf{k}|$, this is not the case on the ``curved'' frequency regions where $k_{3} = {\mathcal O}(1)$ and $k_{2} = {\mathcal O}(|k_{1}|^{r})$, with $0< r \leq 1/2$. In such regions the symbols are unbounded, since as $|k_{1}|\rightarrow \infty$ we have:
\begin{align}
|\widehat{M}_{1}(k_{1},|k_{1}|^{r},1)| \approx |k_{1}|^{r},\quad  |\widehat{M}_{2}(k_{1},|k_{1}|^{r},1)| \approx |k_{1}|,\quad  |\widehat{M}_{3}(k_{1},|k_{1}|^{r},1)| \approx |k_{1}|^{2r}. \label{eq:unbounded:symbol}
\end{align}

Several important properties of the $\widehat{M}_j$'s are immediately obvious:
\begin{enumerate}[i)]
	\item The functions are strongly anisotropic with respect to the dependence on the integers $k_1, k_2$, and $k_3$. This is a consequence of the interplay of the three physical forces governing this
system:
		\begin{itemize}
			\item Coriolis force,
			\item Lorentz force,
			\item Gravity.
\end{itemize}		 

	\item Since the symbols $\widehat{M}_j$ are even the operator $\textbf{M}$ is not anti-symmetric.
\end{enumerate}
These properties of $\widehat{\textbf{M}}$ make the MG equation interesting and challenging mathematically, as well as having a clear physical basis in its derivation from the MHD equations.\vspace{0.2 cm}

Finally, after put our ansatz in (\ref{active_scalar}), we arrive to the following system:
\begin{equation}\label{Omega_general}
\left\{
\begin{array}{rl}
\partial_t \theta(\mathbf{x},t) +\mathbf{u}(\mathbf{x},t)\cdot\nabla\theta(\mathbf{x},t)&=-\epsilon_\kappa \,\Delta\theta(\mathbf{x},t)-\epsilon_\kappa \,\Omega''(x_3)-u_3(\mathbf{x},t)\Omega'(x_3),\\
\mathbf{u}(\mathbf{x},t)&=M[\theta](\mathbf{x},t),\\
\theta(\mathbf{x},0)&=\theta_0(\mathbf{x}),
\end{array}
\right.
\end{equation}
where our initial data $\theta_0$ has zero vertical mean.\vspace{0.2 cm}

\noindent
\textbf{Remark:} Here $\mathbf{x}=(x_1,x_2,x_3)\in \T^2\times\R$, however $\theta(\mathbf{x},t)$ and $\mathbf{u}(\mathbf{x},t)$ are periodic in the three variables.\vspace{0.2 cm}

For the case $\Omega\equiv 0$, the system (\ref{Omega_general}) is again the one widely studied in \cite{Friedlander-Rusin-Vicol,Friedlander-Rusin-Vicol_fractional, Friedlander-Vicol_1, Friedlander-Vicol_3, Friedlander-Vicol_2}. The aim of this paper is to show that the  Cauchy problem for the \textit{non-diffusive} MG equation is  well-posed with respect to perturbations around a specific steady profile $\Omega$, in the topology of a certain Sobolev space.

\subsection{The perturbated \textit{non-diffusive} MG equation}
We fix the perturbation $\theta(\mathbf{x},t):=\Theta(\mathbf{x},t)-\nolinebreak\Omega(x_3)$. Therefore, we obtain the system:
\begin{equation}\label{problema_perturbado}
\left\{
\begin{array}{rl}
\partial_t\theta(\mathbf{x},t) +\mathbf{u}(\mathbf{x},t)\cdot\nabla\theta(\mathbf{x},t) &= -u_3(\mathbf{x},t)\Omega'(x_3),  \\
\mathbf{u}(\mathbf{x},t)&=\phantom{-}M[\theta](\mathbf{x},t),\\
\theta(\mathbf{x},0)&=\phantom{-}\theta_0(\mathbf{x}),
\end{array}
\right.
\end{equation}
with $\mathbf{x}\in\T^3$ and where our initial data $\theta_0$ has zero vertical mean.\vspace{0.2 cm}

What is interesting about this equation is that $M_3$ is a positive operator so we get a mild dissipation effect. This structure will allow us to prove stability. So, just as for the fractional Laplacian, we define the square root of $M_3$ via  Fourier transform as follows:

\begin{defi}\label{sqrt_M3}
The square root of $M_3$ can be defined on functions $f:\T^3\rightarrow \R$ with zero vertical mean as a Fourier multiplier given by the formula:
\begin{equation}\label{square_root_M3}
\widehat{\sqrt{M_3}\,[f]}(\mathbf{k}):= \sqrt{\frac{ k_2^{2}(k_1^{2} + k_2^{2})}{k_3^{2} |\mathbf{k}|^2 +  k_2^{4}}}\,\hat{f}(\mathbf{k}) \qquad \mathbf{k}\in\Z_{\star}^{3}.
\end{equation}
Note that $\widehat{\sqrt{M_3}}(\mathbf{k})$ is not defined on $k_3=0$ since for the self-consistency of the model, we only work with periodic functions with zero vertical mean.
\end{defi}

\subsubsection{The main theorem}
In this work, we are interested in the perturbative regime near the special steady state $\Omega(x_3):=x_3$. The main achievement of the paper is a local existence result for periodic perturbations localized in a suitable section of the frequency space together with a global existence result under an additional size condition over the $H^{5/2}(\mathbb{T}^3)$ norm of the perturbation.\vspace{0.2 cm}

To sum up, we want to consider solutions in $\mathbf{x}\in\T^2\times\mathbb{R}$ and $t\geq 0$ with the structure
$$\Theta(\mathbf{x},t):=\Omega(x_3)+\theta(\mathbf{x},t),$$
for periodic perturbations $\theta(\mathbf{x},0):=\theta_0(\mathbf{x})\hspace*{-0.05 cm}\in \hspace*{-0.05 cm} H^s(\T^3)$ with zero vertical mean and frequency support in $\mathrm{X}\hspace*{-0.05 cm} \subset \hspace*{-0.05 cm} \Z^3$.\linebreak
Then, we prove:
\begin{itemize}
	\item \textbf{Local well-posedness:} \,  If $s>\tfrac{5}{2}$.
	\item \textbf{Global well-posedness:}  If $s>\tfrac{5}{2}$ and $||\theta_0||_{H^{5/2^{+}}}$ is small enough.
	\item \textbf{GWP \& asymptotic stability:}  If $s>\tfrac{7}{2}$ and $||\theta_0||_{H^{7/2^{+}}}$ is small enough.
\end{itemize}
A precise statement of our result is presented as Theorem \ref{main_thm}, where we
also illustrate its proof through a bootstrap argument. Despite the apparent simplicity, understanding the stability of this flow is  non-trivial.\vspace{0.2 cm}

\subsubsection{The ideas behind the proof:}
In order to prove this, first we fix our attention in the study of the stability of the problem, when linearized it around a particular steady state $\Omega(x_3)\equiv x_3$.  The main mechanism of decay can be seen from the linearized equation:
\begin{equation*}
\left\{
\begin{array}{rl}
\partial_t\theta(\mathbf{x},t)&=-M_3[\theta](\mathbf{x},t),  \\
\theta(\mathbf{x},0)&=\phantom{-}\theta_0(\mathbf{x}).
\end{array}
\right.
\end{equation*}
As $\widehat{M_3}(\mathbf{k})$ is a positive operator for $\mathbf{k}\in\Z^3_{\star}\equiv \Z^3\setminus\{k_3=0\}$ there is a unique positive self-adjoint square root operator of $\widehat{M_3}(\mathbf{k})$ on $\Z^3_{\star}$, which we define in (\ref{square_root_M3}). In consequence, the linearized equaiton clearly shows the decay
over time of $\theta(\mathbf{x},t)$, except for  the zero mode in $x_3$. However, we do not have that problem because for self-consistency of the model we restrict to functions that have zero vertical mean.

Hence, the main achievement of the paper is thus to control the nonlinearity, so that it would not destroy the decay provided by the linearized equation. Note that, over the curved frequency regions  where $k_{3} = {\mathcal O}(1)$ and $k_{2} = {\mathcal O}(|k_{1}|^{r})$ with $0< r \leq 1/2$, we have that $\mathbf{u}(\mathbf{x},t)\approx \Lambda\theta(\mathbf{x},t)$ and $M_3[\theta](\mathbf{x},t)\approx \Lambda^{2r}\theta(\mathbf{x},t)$ with $0<r\leq 1/2$ and we can not control and close the estimates at the level of Sobolev spaces. 
But as in \cite{Friedlander-Rusin-Vicol_fractional}, we explore the following observation: if the frequency  support  of $\theta(\mathbf{x},t)$ lies on a suitable section of the Fourier space, then the operator $\mathbf{M}$ behaves like an order zero operator and hence the corresponding velocity $\mathbf{u}(\mathbf{x},t)$ is as smooth as $\theta(\mathbf{x},t)$. 
This enables us to obtain a well-posedness results over the generic setting when no conditions on the Fourier spectrum of the initial perturbation $\theta_0$ are imposed. To be more precise, we consider an appropiate subset $\mathrm{X}\subset\Z^3$ which we will define later, where we can obtain a local well-posedness result for perturbations  $\theta_0(\mathbf{x})$ such that $\text{supp}(\widehat{\theta_0}(\mathbf{k}))\subset \mathrm{X}$.

Under this hypothesis over the initial perturbation, at least morally speaking, our perturbated system behaves like an active scalar of order zero with a damping term:
\begin{equation*}
\left\{
\begin{array}{rl}
\partial_t\theta(\mathbf{x},t)+\mathbf{u}(\mathbf{x},t)\cdot\nabla\theta(\mathbf{x},t)&=-\theta(\mathbf{x},t),  \\
\theta(\mathbf{x},0)&=\phantom{-}\theta_0(\mathbf{x}),
\end{array}
\right.
\end{equation*}
with $\mathbf{u}(\mathbf{x},t)=M[\theta](\mathbf{x},t)$. However, as $\lim_{\alpha\to 0}\Lambda^{\alpha}\theta(\mathbf{x},t)=-\theta(\mathbf{x},t)$, the type of results obtained for the supercritical diffusive MG equation in \cite{Friedlander-Rusin-Vicol_fractional} are expected to have also in our setting.

\subsubsection{Well-prepared initial data $\theta_0$}
In this section we explore the  observation cited above: if the frequency support of $\theta$ lies on a suitable subset of the frecuency space, then the operator $\textbf{M}$ is \textit{mild} when it acts on $\theta$, i.e. it behaves like an order zero operator, and hence the corresponding velocity $\mathbf{u}$ is as smooth as $\theta$. 

This enables us to obtain a well-posedness result over the generic setting when no conditions on the Fourier spectrum of the initial perturbation are imposed. For instance, the local existence and uniqueness of smooth solutions holds for the non-diffusive case, a setting in which we know that for generic initial data the problem is ill-posed in Sobolev spaces.\vspace{0.2 cm}

For the perturbated problem, we are working in the periodic setting $\T^3$ and the frequency space is $\Z^3$. Then, we can define the frequency straight lines $\mathrm{L}(\mathbf{q})$ crossing the origin as the set:
\begin{equation*}\label{plane_p_q}
\mathrm{L}(\mathbf{q}):=\Z^3 \cap \left\lbrace (q_1 k,q_2 k,q_3 k):k\in\Z\right\rbrace \qquad \text{for}\quad  \mathbf{q}:=(q_1,q_2,q_3)\in\mathbb{Q}^3\setminus\{(0,0,0)\}.
\end{equation*}
Now, we will said that $\mathbf{q}\in\mathbb{Q}^3\setminus\{(0,0,0)\}$ is an admissible triple if there exists $\mathfrak{C}>0$ such that $\mathbf{q}\in \mathrm{K}_{\mathfrak{C}}$, where $\mathrm{K}_{\mathfrak{C}}$ is the rational cone defined by
\begin{equation}\label{cone}
\mathrm{K}_{\mathfrak{C}}:=\left\lbrace\mathbf{q}\in\Q^3\setminus\{(0,0,0)\}: |q_1|,|q_3|\leq \mathfrak{C}|q_2|\right\rbrace.
\end{equation}
The next lemma states that $\textbf{M}$ behaves like a zero order operator when it acts on functions with frequency support in $\mathrm{L}(\mathbf{q})$ with $\mathbf{q}\in \mathrm{K}_{\mathfrak{C}}$ for some $\mathfrak{C}>0$. In the rest, we shall make key use of the next properties of $\textbf{M}$.\vspace{0.2 cm}

In the rest of the paper, fixed $\mathfrak{C}>0$ we assume that $\mathrm{X}_{\mathfrak{C}}\in\left\lbrace \{\mathrm{L}(\mathbf{q})\} :q\in \mathrm{K}_{\mathfrak{C}}\right\rbrace$. This is, in the following for $\mathrm{X}_{\mathfrak{C}}$ we will understand one of the previously defined frequency straight lines

\begin{lemma}\label{M_order0}
Let $\mathfrak{C}>0$. For every smooth periodic function $f:\mathbb{T}^3\rightarrow \R$ with zero vertical mean and frequency support in $\mathrm{X}_{\mathfrak{C}}$, there exists a universal constant $\mathfrak{m}^{\star}=\mathfrak{m}^{\star}(\mathfrak{C})>0$ such that:
$$\big|\widehat{M_j[f]}(\mathbf{k})\big|=\big| \widehat{M_j}(\mathbf{k})\, \hat{f}(\mathbf{k})\big|\leq \mathfrak{m}^{\star}\big|\hat{f}(\mathbf{k})\big| \qquad \text{with} \quad j\in\{1,2,3\}$$
for all $\mathbf{k}\in\Z^3$. Moreover, the constant $\mathfrak{m}^{\star}$ blow-up as $\mathfrak{C}$ tends to infinity.
\end{lemma}
\begin{proof}
It is clear that the bound has to be proven only for $\mathbf{k}\in \Z^3_{\star}$, since otherwise we have that $\hat{f}(\mathbf{k})=0$ and the statement holds trivially. Note that $\mathbf{k}\in \mathrm{X}_{\mathfrak{C}}$ implies that $\mathbf{k}=k\cdot\mathbf{q}$ with $\mathbf{q}\in \mathrm{K}_{\mathfrak{C}}$ and some $k\in\Z$. We now consider each of the cases $j\in\{1,2,3\}$.
\begin{itemize}
	\item For $j=1$, a short algebraic computation gives
	$$\big|\widehat{M}_1(\mathbf{k})\big| = \left|\frac{k_2 k_3 |\mathbf{k}|^2 -  k_1 k_2^{2} k_3}{ k_3^{2} |\mathbf{k}|^2 +  k_2^{4}}\right|\leq \left[ (\mathfrak{C}+2\,\mathfrak{C}^3)+\mathfrak{C}^2\right]\frac{q_2^4}{ q_3^{2} |\mathbf{q}|^2 +  q_2^{4}}.$$

	\item Similarly to the previous one, it follows for $j=2$ and $j=3$ that:
	$$\big|\widehat{M}_2(\mathbf{k})\big| = \left|\frac{ k_1 k_3 |\mathbf{k}|^2 +  k_2^{3}k_3}{ k_3^{2} |\mathbf{k}|^2 +  k_2^{4}}\right|\leq \left[ (\mathfrak{C}^2+2\,\mathfrak{C}^4)+\mathfrak{C}\right] \frac{q_2^4}{ q_3^{2} |\mathbf{q}|^2 +  q_2^{4}}$$
	and
$$\big|\widehat{M}_3(\mathbf{k})\big| = \left|\frac{ k_2^{2}(k_1^{2} + k_2^{2})}{ k_3^{2} |\mathbf{k}|^2 +  k_2^{4}}\right|\leq (1+\mathfrak{C}^2)\frac{q_2^4}{ q_3^{2} |\mathbf{q}|^2 +  q_2^{4}}.$$
\end{itemize}
It follows that $|\widehat{M_j}(\mathbf{k})|\leq \mathfrak{m}_j(\mathfrak{C})$ for $j\in\{1,2,3\}$ and taking $\mathfrak{m}^{\star}:=\max\{ \mathfrak{m}_1,\mathfrak{m}_2, \mathfrak{m}_3\}$ concludes the proof.
\end{proof}

Thanks to (\ref{cone}), it is simple to obtain an upper and lower bound for $\widehat{M_3}(\mathbf{k})$ in $\mathrm{X}_{\mathfrak{C}}$. The lower bound will play a key role in the proof of the local and global existence  result. The next lemma gives us this bound.

\begin{lemma}
Let $\mathfrak{C}>0$. For every smooth periodic function $f:\mathbb{T}^3\rightarrow \R$ with zero vertical mean and frequency support in $\mathrm{X}_{\mathfrak{C}}$, there exists a universal constant $\mathfrak{m}_{\star}=\mathfrak{m}_{\star}(\mathfrak{C})>0$ such that:
$$\mathfrak{m}_{\star}\big|\hat{f}(\mathbf{k})\big| \leq   \big|\widehat{M_3[f]}(\mathbf{k})\big|$$
for all $\mathbf{k}\in\Z^3$. Moreover, the constant $\mathfrak{m}_{\star}$ goes to zero as $\mathfrak{C}$ tends to infinity.
\end{lemma}
\begin{proof}
It is clear that the bound has to be proven only for $\mathbf{k}\in \Z^3_{\star}$, since otherwise we have that $\hat{f}(\mathbf{k})=0$ and the statement holds trivially. As $M_3$ is a Fourier multiplier operator, we have that:
$$|\hat{f}(\mathbf{k})|=|\widehat{M_3^{-1}}(\mathbf{k})\, \widehat{M_3[f]}(\mathbf{k})|\leq ||\widehat{M_3^{-1}}||_{L^{\infty}(\mathrm{X}_{\mathfrak{C}})}\,|\widehat{M_3[f]}(\mathbf{k})|.$$
Morover, for  $\mathbf{k}\in \mathrm{X}_{\mathfrak{C}}$ we have the bound
$$\widehat{M_3^{-1}}(\mathbf{k})=\frac{ k_3^{2} |\mathbf{k}|^2 +  k_2^{4}}{ k_2^{2}(k_1^{2} + k_2^{2})}\leq \left[(\mathfrak{C}^2+2\mathfrak{C}^4)+1 \right]\frac{q_2^4}{ q_2^{2}(q_1^{2} + q_2^{2})} $$
from which it follows that $|\hat{f}(\mathbf{k})|\leq \tfrac{1}{\mathfrak{m}_{\star}} |\widehat{M_3[f]}(\mathbf{k})|$ for a suitable constant $\mathfrak{m}_{\star}$.
\end{proof}

As a consequence of the previous lemmas, under the same hypothesis as before we have that the Fourier operator $M_3$ is equivalent to the identity operator in $L^2(\T^3)$. More specifically, there exists a pair of real numbers $0<\mathfrak{m}_{\star}\leq \mathfrak{m}^{\star}$ such that:
$$\mathfrak{m}_{\star}||f||_{L^2(\T^3)}\leq ||M_3[f]||_{L^2(\T^3)}\leq \mathfrak{m}^{\star}||f||_{L^2(\T^3)}.$$

\begin{cor}\label{M3_order_zero}
Let $\mathfrak{C}>0$. For every smooth periodic function $f:\mathbb{T}^3\rightarrow \R$ with zero vertical mean and frequency support in $\mathrm{X}_{\mathfrak{C}}$, there exists positive constants $\mathfrak{m}_{\star}$ and $\mathfrak{m}^{\star}$ such that:
$$\mathfrak{m}_{\star}\big|\hat{f}(\mathbf{k})\big| \leq  \big|\widehat{M_3[f]}(\mathbf{k})\big|\leq \mathfrak{m}^{\star}\big|\hat{f}(\mathbf{k})\big| \qquad \text{for all} \quad \mathbf{k}\in\Z^3.$$
\end{cor}

The key point of the result of well-posedness is the fact that we only work with frequency localized initial perturbations. As we will see later in the proof, to prove that the perturbation does not leave the region of the frequency space where the operator $\textbf{M}$ behaves like a zero order operator,  the sets closed under addition will play a crucial role.

\subsubsection{Closed sets under addition}
A set $\mathrm{X}$ is closed under addition $+:\mathrm{X}\times \mathrm{X}\rightarrow \mathrm{X}$ if for all $a,b\in \mathrm{X}$ we have that $a+b\in \mathrm{X}$. In other words, performing the binary operation on any two elements of the set always gives you back something that is also in the set.\\

\begin{lemma}\label{Fourier_properties}
Fixed $\mathfrak{C}\in \mathbb{Q}$ and $\mathrm{X}_{\mathfrak{C}}$.
For every pair of smooth periodic function $f,g:\mathbb{T}^3\rightarrow \R$ with frequency support in $\mathrm{X}_{\mathfrak{C}}$ we have that:
\begin{itemize}
	\item $\text{supp}\left(\widehat{f\, g}\right)\subset \mathrm{X}_{\mathfrak{C}}$.
	\item $\text{supp}\left(\widehat{f\pm g}\right)\subset \mathrm{X}_{\mathfrak{C}}$.
	\item $\text{supp}\left(\widehat{M_j[f]}\right)\subset \mathrm{X}_{\mathfrak{C}}$ for all $j\in\{1,2,3\}$.
\end{itemize}
\begin{proof}The proof is an immediate consequence of the properties of the Fourier transform:
\begin{itemize}
	\item Clearly $\text{supp}(\widehat{f\, g})=\text{supp}(\hat{f}\ast \hat{g})\subset \text{supp}(\hat{f})+\text{supp}(\hat{g})\subset \mathrm{X}_{\mathfrak{C}},$ since $\mathrm{X}_{\mathfrak{C}}$ is closed under addition.
	
	\item Note that $\text{supp}(\widehat{f\pm g})=\text{supp}(\hat{f}\pm \hat{g})\subset \text{supp}(\hat{f})\cup \text{supp}(\hat{g})\subset  \mathrm{X}_{\mathfrak{C}}.$
	
	\item As $\textbf{M}$ is a Fourier multiplier, we have: $\text{supp}(\widehat{M_j[f]})=\text{supp}(\widehat{M_j}\hat{f})\subset\text{supp}(\widehat{M_j})\cap \text{supp}(\hat{f})\subset  \mathrm{X}_{\mathfrak{C}}.$
\end{itemize}
\vspace{-0.4 cm}
\end{proof}
\end{lemma}

\section{Local existence for frequency-localized initial data}\label{Sec_3}
The main result of this section is:
\begin{thm}\label{local_existence} Fixed $\mathfrak{C}\in \mathbb{Q}$ and $\mathrm{X}_{\mathfrak{C}}$. Assume that $\theta_0\in H^s(\T^3)$  with $s>5/2$ has zero vertical mean and satisfies that $\text{supp}(\widehat{\theta_0})\subset \mathrm{X}_{\mathfrak{C}}$. Then, there exists a time $T>0$ and a unique smooth solution $$\theta\in L^{\infty}(0,T;H^s(\T^3))$$
of the Cauchy problem (\ref{problema_perturbado}) such that $\text{supp}(\widehat{\theta}(t))\subset \mathrm{X}_{\mathfrak{C}}$ for all $t\in[0,T)$.
\end{thm}

Before that, the goal is to prove the existence of smooth solutions to the scalar \textit{linear} equation:
\begin{equation}\label{linear_hyperdisipation}
\left\{
\begin{array}{rl}
\partial_t\theta(\mathbf{x},t)+\mathbf{v}(\mathbf{x},t)\cdot\nabla\theta(\mathbf{x},t) &= -M_3[\theta](\mathbf{x},t)  \\
\theta(\mathbf{x},0)&=\phantom{-}\theta_0(\mathbf{x})
\end{array}
\right.
\end{equation}
where the initial datum $\theta_0$ and the given divergence-free drift velocity field $\mathbf{v}$ satisfies that:
\begin{itemize}
	\item $\text{supp}(\widehat{\theta_0})\subset \mathrm{X}_{\mathfrak{C}}$.
	\item $\text{supp}(\widehat{\mathbf{v}}(t))\subset \mathrm{X}_{\mathfrak{C}}$ for all $t\in[0,T)$ for a positive time $T$.
\end{itemize}
The main result is:
\begin{thm}\label{local_existence_hyperdisipation}
Let $s>3/2$. Given $\theta_0\in H^s(\T^3)$ and a divergence-free vector field $\mathbf{v}\in L^{\infty}(0,T;H^s(\T^3))$ satisfying the above conditions. Then, there exists a unique smooth solution of (\ref{linear_hyperdisipation}) such that:
\begin{equation}\label{smoothness_class}
\theta\in L^{\infty}(0,T;H^s(\T^3)).
\end{equation}
Moreover, we have that $\text{supp}(\hat{\theta}(t))\subset  \mathrm{X}_{\mathfrak{C}}$ for all $t\in[0,T).$
\end{thm}

\begin{proof}[Proof of Theorem \ref{local_existence_hyperdisipation}]
Following the arguments of \cite{Friedlander-Rusin-Vicol_fractional}, we regularize (\ref{linear_hyperdisipation}) with hyper-dissipation as:
\begin{equation}\label{linear_hyperdisipation_epsilon}
\left\{
\begin{array}{rl}
\partial_t\theta^{\epsilon}(\mathbf{x},t)+\mathbf{v}(\mathbf{x},t)\cdot\nabla\theta^{\epsilon}(\mathbf{x},t)-\epsilon\,\Delta\theta^{\epsilon}(\mathbf{x},t) &= -M_3[\theta^{\epsilon}](\mathbf{x},t)  \\
\theta^{\epsilon}(\mathbf{x},0)&=\phantom{-}\theta_0(\mathbf{x})
\end{array}
\right.
\end{equation}
for $\epsilon\in(0,1]$ and finally we pass to the limit $\epsilon\to 0$ in order to obtain a solution of the original system. 

On one hand, since $\mathbf{v}$ is smooth and divergence-free, it follows from the De Giorgi techniques (see  \cite{Friedlander-Rusin-Vicol_fractional} or \cite{Seregin-Silvestre-Sverak-Zlatos}) that there exists a \textit{unique} global \textit{smooth} solution $\theta^{\epsilon}$ of (\ref{linear_hyperdisipation_epsilon}) with
$$\theta^{\epsilon}\in L^{\infty}(0,T;H^s(\T^3))\cap \epsilon L^2(0,T;H^{s+1}(\T^3)).$$

On the other hand, we proceed to construct a solution $\theta^{\epsilon}$ of (\ref{linear_hyperdisipation_epsilon}) which has the desired frequency support property and belongs to the smooth category. Then, by the uniqueness of strong solutions, we pass to the limit $\epsilon\rightarrow 0$ and obtain a solution with desired properties. We consider the following iterative scheme:
\begin{equation}\label{iteracion_1_epsilon}
\left\{
\begin{array}{rl}
\partial_t\theta_1^{\epsilon}(\mathbf{x},t)-\epsilon\,\Delta\theta_{1}^{\epsilon}(\mathbf{x},t) &= -M_3[\theta_1^{\epsilon}](\mathbf{x},t)  \\
\theta_1^{\epsilon}(\mathbf{x},0)&=\phantom{-}\theta_0(\mathbf{x})
\end{array}
\right.
\end{equation}
and
\begin{equation}\label{iteracion_n_epsilon}
\left\{
\begin{array}{rl}
\partial_t\theta_{n+1}^{\epsilon}(\mathbf{x},t) +\mathbf{v}(\mathbf{x},t)\cdot\nabla\theta_{n}^{\epsilon}(\mathbf{x},t) -\epsilon\,\Delta\theta_{n+1}^{\epsilon}(\mathbf{x},t)&= -M_3[\theta_{n+1}^{\epsilon}](\mathbf{x},t)  \\
\theta_{n+1}^{\epsilon}(\mathbf{x},0)&=\phantom{-}\theta_0(\mathbf{x})
\end{array}
\right.
\end{equation}
for all $n\geq 1$. We note that the solutions of (\ref{iteracion_1_epsilon}) and (\ref{iteracion_n_epsilon}) respectively, may be written explicitly using the Duhamel's formula:
\begin{align*}
\theta_1^{\epsilon}(\mathbf{x},t)&=e^{-(\epsilon(-\Delta)+M_3)\,t}\,\theta_0(\mathbf{x}),\\
\theta_{n+1}^{\epsilon}(\mathbf{x},t)&=e^{-(\epsilon(-\Delta)+M_3)\,t}\,\theta_0(\mathbf{x})+\int_{0}^{t}e^{-(\epsilon(-\Delta)+M_3)\,(t-\tau)}\, (\mathbf{v}(\mathbf{x},\tau)\cdot\nabla\theta_{n}^{\epsilon}(\mathbf{x},\tau))\,d\tau.
\end{align*}
Since $e^{-(\epsilon(-\Delta)+M_3)\,t}$  is given explicitly by the Fourier multiplier with non-zero symbol $e^{-(\epsilon|\mathbf{k}|+\widehat{M_3}(\mathbf{k}))\,t}$, this operator does not alter the frequency support of the function on which it acts. Therefore, it follows directly from our assumption on the frequency support of $\theta_0$ that $\text{supp}(\widehat{\theta_1^{\epsilon}}(t))\subset \mathrm{X}_{\mathfrak{C}}$ for all $t\in[0,T).$

Now, we proceed inductively and note that if $\text{supp}(\widehat{\theta_n^{\epsilon}}(t))\subset \mathrm{X}_{\mathfrak{C}}$ for all $t\in[0,T)$. Then, by our assumption on the frequency support of $\mathbf{v}$ and Lemma \ref{Fourier_properties} we also have $\text{supp}(\widehat{\mathbf{v}\cdot\nabla\theta_n^{\epsilon}}(t))\subset \mathrm{X}_{\mathfrak{C}}$ for all $t\in[0,T)$. Hence, we obtain that  $\text{supp}(\widehat{\theta_{n+1}^{\epsilon}}(t))\subset \mathrm{X}_{\mathfrak{C}}$ for all $t\in[0,T)$ concluding the proof of the induction step. This proves that the frequency support of all the iterates $\theta_n^{\epsilon}(t)$ lies on $\mathrm{X}_{\mathfrak{C}}$ for all $t\in[0,T)$.

Thus, it is left to prove that the sequence $\{\theta_n^{\epsilon}\}_{n\geq 1}$ converges to a function $\theta^{\epsilon}$ which lies in the smoothness class (\ref{smoothness_class}). Note that there is no cancellation of the highest order term in the nonlinearity. However, since (at least for now) $\epsilon\in(0,1]$ is fixed, we may use the full smoothing power of the Laplacian.\vspace{0.2 cm}

To prove it, for all $n\geq 1$ we define:
$$\mathcal{R}_n(t):=\sup_{\tau\in[0,t]}||\Lambda^s\theta_n^{\epsilon}||_{L^2}^2(\tau)+\int_{0}^{t} ||\sqrt{M_3}\,[\Lambda^{s}\theta_n^{\epsilon}]||_{L^2}^2(\tau)\,d\tau+\epsilon \int_{0}^{t} ||\Lambda^{s+1}\theta_n^{\epsilon}||_{L^2}^2(\tau)\,d\tau.$$
Moreover, as the frequency support of all the iterates $\theta_n^{\epsilon}$ lies on $\mathrm{X}_{\mathfrak{C}}\subset \mathrm{K}_{\mathfrak{C}}$, using Corollary \ref{M3_order_zero} we have:
$$||\sqrt{M_3}\,[\Lambda^{s}\theta_n^{\epsilon}]||_{L^2}^2(\tau)\approx ||\Lambda^{s}\theta_n^{\epsilon}||_{L^2}^2(\tau).$$

In the first step, note that from (\ref{iteracion_1_epsilon}) it follows that for any $t\in(0,T]$ we obtain that $\mathcal{R}_1(t)\leq 2\, ||\Lambda^s\theta_0||_{L^2}^2$. We proceed inductively and assume that there exists a time $T^{\star}\in(0,T]$ such that $\mathcal{R}_{n}(T^{\star})\leq  2\, ||\Lambda^s\theta_0||_{L^2}^2$. Here, we show that if $T^{\star}$
 is chosen appropriately, in terms of $\theta_0, \mathbf{v}$ and $\epsilon$, we have $\mathcal{R}_{n+1}(T^{\star})\leq  2\, ||\Lambda^s\theta_0||_{L^2}^2$ \nolinebreak too.

From (\ref{iteracion_n_epsilon}), the divergence-free velocity field $\nabla\cdot\mathbf{v}=0$, integration by parts and the fact that $s>3/2$ which makes $H^s$ an algebra, we obtain:
$$\tfrac{1}{2}\partial_t ||\Lambda^s\theta_{n+1}^{\epsilon}||_{L^2}^2(t)+\,||\sqrt{M_3}\,[\Lambda^s\theta_{n+1}^{\epsilon}]||_{L^2}^2(t)+\epsilon\, ||\Lambda^{s+1}\theta_{n+1}^{\epsilon}||_{L^2}^2(t)\leq \frac{||\Lambda^s(\mathbf{v}\,\theta_n^{\epsilon})||_{L^2}^2(t)}{2\,\epsilon} +\frac{\epsilon\,||\Lambda^{s+1}\theta_n^{\epsilon}||_{L^2}^2(t)}{2}$$
and in consequence, as $\mathcal{R}_{n}(T^{\star})\leq  2\, ||\Lambda^s\theta_0||_{L^2}^2$ we have proved that:
\begin{align*}
\mathcal{R}_{n+1}(T^{\star})&\leq ||\Lambda^s\theta_0||_{L^2}^2+\frac{C_s}{\epsilon}\int_{0}^{T^{\star}}||\Lambda^s\mathbf{v}||_{L^2}^2(\tau)\,||\Lambda^{s}\theta_n^{\epsilon}||_{L^2}^2(\tau)\,d\tau\\
&\leq ||\Lambda^s\theta_0||_{L^2}^2+\frac{2 C_s T^{\star}}{\epsilon}||\mathbf{v}||_{L^{\infty}(0,T;H^s)}^2\,||\Lambda^{s}\theta_0||_{L^2}^2.
\end{align*} 
Hence, if we let 
\begin{equation}\label{valor_T1}
T^{\star}\leq \frac{\epsilon}{2 C_s ||\mathbf{v}||_{L^{\infty}(0,T;H^s)}^2}
\end{equation}
we have that $\mathcal{R}_{n+1}(T^{\star})\leq 2||\Lambda^{s}\theta_0||_{L^2}^2$. Since $T^{\star}$ is independent of $n$, it is clear that the inductive argument may be carried through, and hence $\mathcal{R}_n(T^{\star})\leq 2 ||\Lambda^{s}\theta_0||_{L^2}^2$ for all $n\geq 1$.\vspace{0.2 cm}

The second step is the passage to the limit in $n$. Taking the difference of two iterates:
\begin{equation}\label{difference_two}
\left\{
\begin{array}{rl}
\partial_t(\theta_{n+1}^{\epsilon}-\theta_{n}^{\epsilon})(\mathbf{x},t) +\mathbf{v}(\mathbf{x},t)\cdot\nabla(\theta_{n}^{\epsilon}-\theta_{n-1}^{\epsilon})(\mathbf{x},t) -\epsilon\,\Delta(\theta_{n+1}^{\epsilon}-\theta_{n}^{\epsilon})(\mathbf{x},t)&= -M_3[\theta_{n+1}^{\epsilon}-\theta_{n}^{\epsilon}](\mathbf{x},t)  \\
(\theta_{n+1}^{\epsilon}-\theta_{n}^{\epsilon})(\mathbf{x},0)&=\phantom{-}0
\end{array}
\right.
\end{equation}
for all $n\geq 2$. Similarly to the above, it follows from (\ref{difference_two}) that
\begin{align*}
\hspace*{-0.15 cm}\widetilde{\mathcal{R}}_n(t)&:=\sup_{\tau\in[0,t]}||\Lambda^s(\theta_{n+1}^{\epsilon}-\theta_{n}^{\epsilon})||_{L^2}^2(\tau)+\int_{0}^{t} ||\sqrt{M_3}\,[\Lambda^{s}(\theta_{n+1}^{\epsilon}-\theta_{n}^{\epsilon})]||_{L^2}^2(\tau)\,d\tau+\epsilon \int_{0}^{t} ||\Lambda^{s+1}(\theta_{n+1}^{\epsilon}-\theta_{n}^{\epsilon})||_{L^2}^2(\tau)\,d\tau\\
&\leq \frac{C_s}{\epsilon}\int_{0}^{t} ||\Lambda^s\mathbf{v}||_{L^2}^2(\tau) ||\Lambda^s(\theta_{n}^{\epsilon}-\theta_{n-1}^{\epsilon})||_{L^2}^2(\tau)\,d\tau\leq \frac{t\, C_s}{\epsilon} ||\mathbf{v}||_{L^{\infty}(0,T;H^s)}^2 \widetilde{\mathcal{R}}_{n-1}(t) \hspace{1.8 cm} \text{for all} \quad n\geq 2.
\end{align*}
In particular, due to our choice of $T^{\star}\in(0,T]$ on (\ref{valor_T1}) we have that $\widetilde{\mathcal{R}}_n(T^{\star})\leq\frac{1}{2}\widetilde{\mathcal{R}}_{n-1}(T^{\star})$, which implies that the sequence $\{\theta_n^{\epsilon}\}_{n\geq 1}$ is not only bounded,  we actually have a contraction in
\begin{equation}\label{smooth_class_T1}
L^{\infty}(0,T^{\star};H^s(\T^3))\cap \epsilon L^2(0,T^{\star};H^{s+1}(\T^3)).
\end{equation}
Hence there exists a limiting function $\theta_n^{\epsilon}\rightarrow \theta^{\epsilon}$ in the category (\ref{smooth_class_T1}). In addition, since for every $n\geq 1$ we have $\text{supp}(\widehat{\theta_n^{\epsilon}}(t))\subset \mathrm{X}_{\mathfrak{C}}$ and the set $\mathrm{X}_{\mathfrak{C}}$ is closed, we automatically obtain that $\text{supp}(\widehat{\theta^{\epsilon}}(t))\subset \mathrm{X}_{\mathfrak{C}}$.

To show that $\theta^{\epsilon}$ may be continued in (\ref{smooth_class_T1}) up to time $T$, we note that $||\Lambda^s \theta^{\epsilon}||_{L^2}^2(T^{\star})\leq 2 ||\Lambda^s \theta_0||_{L^2}^2$ thanks to the fact that $\mathcal{R}_n(T^{\star})\leq 2 ||\Lambda^s \theta_0||_{L^2}^2$ for all $n\geq 1$. Hence, repeating the above argument with initial condition $\theta^{\epsilon}(T^{\star})$, we obtain a solution $\theta^{\epsilon}\in L^{\infty}(0,2\,T^{\star};H^s(\T^3))\cap \epsilon L^2(0,2\,T^{\star};H^{s+1}(\T^3))$
whit the bound $||\Lambda^s\theta^{\epsilon}||_{L^2}^2(2\,T^{\star})\leq 2||\Lambda^s\theta^{\epsilon}||_{L^2}^2(T^{\star})\leq 4 ||\Lambda^s \theta_0||_{L^2}^2$.

The above argument may be extended iteratively, thereby concluding the construction of the solution $\theta^{\epsilon}$ in the category (\ref{smoothness_class}).\vspace{0.2 cm}

In order to close the proof we need to pass to the limit as $\epsilon\rightarrow 0$. By construction we have that $\theta^{\epsilon}$ is uniformly bounded, with respect to $\epsilon$ in $L^{\infty}(0,T;H^s(\T^3))$, and from (\ref{linear_hyperdisipation_epsilon}) we obtain that $\partial_t\theta^{\epsilon}$ is uniformly bounded, with respect to $\epsilon$ in $L^{\infty}(0,T;H^{s-2}(\T^3)) \cap L^2(0,T;H^{s-2}(\T^3))$. In particular, from the uniform bounds (with respect to $\epsilon$) of $\theta^{\epsilon}$ and $\partial_t\theta^{\epsilon}$ in the corresponding norms, one can use the Banach-Alaoglu theorem and the Aubin-Lions's compactness lemma (see, e.g. \cite{Lions} or \cite{Temam}) to justify that one can extract a subsequence of $\theta^{\epsilon}$ and $\partial_t\theta^{\epsilon}$ (using the same index for simplicity) as $\epsilon\to 0$ and elements $\theta$ and $\partial_t\theta$, such that:
\begin{itemize}
	\item \quad $\theta^{\epsilon}\rightarrow \theta$  strongly in $C(0,T;H^{s-1}(\T^3))$.
	\item \quad $\partial_t\theta^{\epsilon}\rightharpoonup \partial_t\theta$ weakly in $L^2(0,T;H^{s-2}(\T^3))$.
	\item \quad $\partial_t\theta^{\epsilon}\overset{\ast}{\rightharpoonup} \partial_t\theta$ weakly-$\ast$ in $L^{\infty}(0,T;H^{s-2}(\T^3))$.
\end{itemize}
Now, from (\ref{linear_hyperdisipation_epsilon}) we have that $\partial_t\theta^{\epsilon} \rightarrow -M_3[\theta]-\mathbf{v}\cdot\nabla\theta$ in $C(0,T;H^{s-2}(\T^3))$. Moreover, as $\theta^{\epsilon}\rightarrow \theta$ in $C(0,T;H^{s-1}(\T^3))$, the distribution limit of $\partial_t\theta^{\epsilon}$ must be $\partial_t\theta$ for the closed graph theorem \cite{Brezis}. In consequence, since the evolution is linear and $s$ is large enough, it follows that this limiting function is the unique smooth solution of (\ref{linear_hyperdisipation}) which lies in $L^{\infty}(0,T;H^s(\T^3))$. Lastly, since for every $\epsilon\in(0,1]$ we have $\text{supp}(\widehat{\theta^{\epsilon}})\subset \mathrm{X}_{\mathfrak{C}}$, and since $\mathrm{X}_{\mathfrak{C}}$ is closed, we obtain that the limiting function also has the desired support property,  i.e. $\text{supp}(\widehat{\theta})\subset\mathrm{X}_{\mathfrak{C}}$, which concludes the proof of the theorem.
\end{proof}

The main difficulty in the proof of the previous theorem is the construction of an iteration scheme which is both  suitble for energy estimates and preserves the feature that in each iteration step the frequency support of the approximation lies on $\mathrm{X}_{\mathfrak{C}}$. Now, we are ready to prove the main result of this section.

\begin{proof}[Proof of Theorem \ref{local_existence}]
In order to construct the local in time solution $\theta$ with frequency support in $\mathrm{X}_{\mathfrak{C}}$, we consider the sequence of approximations $\{\theta_n\}_{n\geq 1}$ given by the solutions of
\begin{equation}\label{iteracion_1}
\left\{
\begin{array}{rl}
\partial_t\theta_1(\mathbf{x},t) &= -M_3[\theta_1](\mathbf{x},t)  \\
\theta_1(\mathbf{x},0)&=\phantom{-}\theta_0(\mathbf{x})
\end{array}
\right.
\end{equation}
and
\begin{equation}\label{iteracion_n}
\left\{
\begin{array}{rl}
\partial_t\theta_n(\mathbf{x},t) +\mathbf{u}_{n-1}(\mathbf{x},t)\cdot\nabla\theta_{n}(\mathbf{x},t) &= -M_3[\theta_n](\mathbf{x},t)  \\
\mathbf{u}_{n-1}(\mathbf{x},t)&=\phantom{-}M[\theta_{n-1}](\mathbf{x},t)\\
\theta_n(\mathbf{x},0)&=\phantom{-}\theta_0(\mathbf{x})
\end{array}
\right.
\end{equation}
for all $n\geq 2$. One may solve (\ref{iteracion_1}) explicitly in the frequency space as
$\widehat{\theta_1}(\mathbf{k},t)=e^{-\widehat{M_3}(\mathbf{k})\,t}\,\widehat{\theta_0}(\mathbf{k})$ for $\mathbf{k}\in\Z^3$. Hence, it is clear that $\text{supp}(\widehat{\theta_1}(t))\subset \text{supp}(\widehat{\theta_0})\subset \mathrm{X}_{\mathfrak{C}}$ for all $t\geq 0$.  Moreover, we have that:
$$ ||\Lambda^s\theta_1||_{L^2}^2(t)+2\int_{0}^{t}||\sqrt{M_3}\,[\Lambda^s\theta_1]||_{L^2}^2(\tau)\,d\tau=||\Lambda^s\theta_0||_{L^2}^2 \qquad \text{for all} \quad t\geq 0.$$
In particular, fixed $T>0$, we obtain the bound:
$$ ||\Lambda^s\theta_1||_{L^{\infty}(0,T;L^2)}^2+||\sqrt{M_3}\,[\Lambda^s\theta_1]||_{L^{2}(0,T;L^2)}^2\leq  2\, ||\Lambda^s\theta_0||_{L^2}^2.$$
In order to solve (\ref{iteracion_n}) we appeal to Theorem \ref{local_existence_hyperdisipation}. Indeed, by the inductive assumption we have that $\theta_{n-1} \in L^\infty(0,T;H^s)$  and also that $\text{supp}(\widehat{\theta_{n-1}}(t))\subset \mathrm{X}_{\mathfrak{C}}$ for all $t\in[0,T)$. Hence, as $\mathbf{u}_{n-1}\equiv M[\theta_{n-1}]$ by applying Lemma \ref{M_order0} we have that $\mathbf{u}_{n-1} \in L^\infty(0,T;H^s)$ and by Lemma \ref{Fourier_properties} we have $\text{supp}(\widehat{\mathbf{u}_{n-1}}(t))\subset \mathrm{X}_{\mathfrak{C}}$ for $t\in[0,T)$. Therefore, all the conditions of Theorem \ref{local_existence_hyperdisipation} are satisfied, by letting $\mathbf{v} = \mathbf{u}_{n-1}$, and there exists a unique solution $\theta_n \in L^\infty(0,T;H^s) $ of (\ref{iteracion_n}), such that $\text{supp}(\widehat{\theta_n}(t)) \subset \mathrm{X}_{\mathfrak{C}}$ for $t\in [0,T)$. Moreover, using that $\theta_0$ has zero vertical mean on $\T^3$ the sequence $\{\theta_n\}_{n\geq 1}$ satisfies  the same by construction.\vspace{0.2 cm}

To prove that the sequence $\{\theta_n\}_{n\geq 1}$ converges, we first prove that it is bounded. To do it, we assume inductively that the following bound holds for all $1\leq j\leq n-1$ and proceed to prove that it holds for $j=n$.
\begin{equation}\label{induction_assumption}
||\Lambda^s\theta_j||_{L^{\infty}(0,T;L^2)}^2+||\sqrt{M_3}\,[\Lambda^s\theta_j]||_{L^{2}(0,T;L^2)}^2\leq 2\,||\Lambda^s\theta_0||_{L^2}^2.
\end{equation}
Applying $\Lambda^s$ to (\ref{iteracion_n}) and taking an $L^2$ inner product with $\Lambda^s\theta_n$ we obtain:
\begin{align*}
\tfrac{1}{2}\partial_t||\Lambda^s\theta_n||_{L^2}^2(t)+||\sqrt{M_3}\,[\Lambda^s \theta_n]||_{L^2}^2(t)&\leq ||\Lambda^s\theta_n||_{L^2}(t)\,||\left[\mathbf{u}_{n-1}\cdot\nabla,\Lambda^s\right]\theta_n||_{L^2}(t)\\
&\lesssim ||\Lambda^s\theta_n||_{L^2}(t)\left(||\nabla\mathbf{u}_{n-1}||_{L^{\infty}}||\Lambda^s\theta_n||_{L^2}+ ||\Lambda^s\mathbf{u}_{n-1}||_{L^2}||\nabla\theta_n||_{L^{\infty}}\right)(t)
\end{align*} 
and for $s>5/2$ we have that:
\begin{equation}\label{lwp_part1}
\tfrac{1}{2}\partial_t||\Lambda^s\theta_n||_{L^2}^2(t)+||\sqrt{M_3}\,[\Lambda^s \theta_n]||_{L^2}^2(t)\lesssim  ||\Lambda^s\theta_n||_{L^2}^2(t) \,||\Lambda^{s}\mathbf{u}_{n-1}||_{L^2}(t).
\end{equation}
Above, we have used the fact that $\nabla\cdot\mathbf{u}_{n-1}=0$ in order to write the commutator estimate and  the Sobolev embedding $L^{\infty}(\T^3)\hookrightarrow H^{3/2^{+}}(\T^3)$. Since $\mathbf{u}_{n-1}$ is obtained from $\theta_{n-1}$ by a bounded Fourier multiplier (cf. Lemma \ref{M_order0}) there exists a positive constant $\mathfrak{m}^{\star}$ such that:
\begin{equation}\label{lwp_part2}
||\Lambda^s\mathbf{u}_{n-1}||_{L^2}(t)\leq \mathfrak{m}^{\star}||\Lambda^s\theta_{n-1}||_{L^2}(t).
\end{equation}
In consequence, putting together (\ref{lwp_part1}) and (\ref{lwp_part2}) for $s>5/2$ we have proved that:
\begin{align*}
\tfrac{1}{2}\partial_t||\Lambda^s\theta_n||_{L^2}^2(t)+||\sqrt{M_3}\,[\Lambda^s \theta_n]||_{L^2}^2(t)&\leq C_s\,\mathfrak{m}^{\star}\, ||\Lambda^s\theta_n||_{L^2}^2(t)\,||\Lambda^s\theta_{n-1}||_{L^2}(t).
\end{align*}

By Corollary (\ref{M3_order_zero}) there exists two positive constants such that $\mathfrak{m}_{\star}\leq \widehat{M_3}(\mathbf{k})\leq \mathfrak{m}^{\star}$ for all $k\in\Z^3_{\star}$. Hence, applying H\"older's inequality we get:
\begin{align*}
\tfrac{1}{2}\partial_t||\Lambda^s\theta_n||_{L^2}^2(t)+||\sqrt{M_3}\,[\Lambda^s \theta_n]||_{L^2}^2(t)&\leq \frac{C_s\,\mathfrak{m}^{\star}}{\sqrt{\mathfrak{m}_{\star}}}\, ||\Lambda^s\theta_n||_{L^2}(t)\,||\Lambda^s\theta_{n-1}||_{L^2}(t)\,||\sqrt{M_3}\,[\Lambda^s\theta_n]||_{L^2}(t)\\
&\leq \frac{1}{2}\left(\frac{C_s\,\mathfrak{m}^{\star}}{\sqrt{\mathfrak{m}_{\star}}}\right)^2\, ||\Lambda^s\theta_n||_{L^2}^2(t)\,||\Lambda^s\theta_{n-1}||_{L^2}^2(t)+\frac{1}{2}||\sqrt{M_3}\,[\Lambda^s\theta_n]||_{L^2}^2(t).
\end{align*}
Using the inductive assumption (\ref{induction_assumption}), it follows for $t\in[0,T)$ that:
$$\partial_t||\Lambda^s\theta_n||_{L^2}^2(t)+||\sqrt{M_3}\,[\Lambda^s \theta_n]||_{L^2}^2(t)\leq \left(\frac{C_s\,\mathfrak{m}^{\star}}{\sqrt{\mathfrak{m}_{\star}}} ||\Lambda^s\theta_0||_{L^2}\right)^2\, ||\Lambda^s\theta_n||_{L^2}^2(t)$$
and applying Gr\"onwall's inequality, we arrive to:
$$||\Lambda^s\theta_n||_{L^2}^2(t)+\int_{0}^{t}||\sqrt{M_3}\,[\Lambda^s\theta_n]||_{L^2}^2(\tau)\,d\tau\leq \exp{\left[\left(\frac{C_s\,\mathfrak{m}^{\star}}{\sqrt{\mathfrak{m}_{\star}}} ||\Lambda^s\theta_0||_{L^2}\right)^2\, t\right]}||\Lambda^s\theta_0||_{L^2}^2.$$
Therefore, taking 
\begin{equation}\label{value_T_1}
T\leq \frac{\log 2 }{\left(\frac{C_s\,\mathfrak{m}^{\star}}{\sqrt{\mathfrak{m}_{\star}}} ||\Lambda^s\theta_0||_{L^2}\right)^2} 
\end{equation}
we obtain that (\ref{induction_assumption}) holds for $j=n$ and so by induction it holds for all $j\geq 1$. This shows that the sequence $\{\theta_n\}_{n\geq 1}$ is uniformly bounded in $L^{\infty}(0,T;H^s).$\vspace{0.2 cm}

Moreover, we may show that the sequence $\{\theta_n\}_{n\geq 1}$ is Cauchy in $L^{\infty}(0,T;H^{s-1})$. To see this, we consider the difference of two iterates $\widetilde{\theta}_n:=\theta_n-\theta_{n-1}$. It follows from (\ref{iteracion_n}) that $\widetilde{\theta}_n$ is a solution of:
\begin{equation}\label{difference_iterates}
\left\{
\begin{array}{rl}
\partial_t\widetilde{\theta}_n(\mathbf{x},t) +\mathbf{u}_{n-1}(\mathbf{x},t) \cdot\nabla\widetilde{\theta}_n(\mathbf{x},t) +\widetilde{\mathbf{u}}_{n-1}(\mathbf{x},t) \cdot\nabla\theta_{n-1}(\mathbf{x},t) &= -M_3[\widetilde{\theta}_n](\mathbf{x},t)   \\
\mathbf{u}_{n-1}(\mathbf{x},t) &=\phantom{-}M[\theta_{n-1}](\mathbf{x},t) \\
\widetilde{\theta}_n(\mathbf{x},0) &=\phantom{-}0
\end{array}
\right.
\end{equation}
for all $n\geq 3$, where $\widetilde{\mathbf{u}}_{n}(\mathbf{x},t) :=M[\widetilde{\theta}_{n}](\mathbf{x},t)$. Applying $\Lambda^{s-1}$ to (\ref{difference_iterates}), taking an $L^2$ inner product with $\Lambda^{s-1}$ and using that $\nabla\cdot\mathbf{u}_{n-1}=0$, we arrive to:
\begin{equation}\label{lwp_parte_A}
\begin{aligned}
\tfrac{1}{2}\partial_t||\Lambda^{s-1}\widetilde{\theta}_n||_{L^2}^2(t)+||\sqrt{M_3}\,[\Lambda^{s-1}\widetilde{\theta}_n]||_{L^2}^2(t) &\leq   ||\Lambda^{s-1}\widetilde{\theta}_n||_{L^2}(t)\,||\left[\mathbf{u}_{n-1}\cdot\nabla,\Lambda^{s-1}\right]\widetilde{\theta}_n||_{L^2}(t)\\
&\quad + ||\Lambda^{s-1}\widetilde{\theta}_n||_{L^2}(t)\,||\Lambda^{s-1}(\widetilde{\mathbf{u}}_{n-1}\cdot\nabla\theta_{n-1})||_{L^2}(t). 
\end{aligned}
\end{equation}
Now, applying the Sobolev embeddings into the product and commutator estimate given by Lemma \ref{product_estimate} and Lemma \ref{commutator_estimate}, we obtain for $s>5/2$ that:
\begin{align*}
\bullet \quad ||\left[\mathbf{u}_{n-1}\cdot\nabla,\Lambda^{s-1}\right]\widetilde{\theta}_n||_{L^2}&\lesssim ||\nabla \mathbf{u}_{n-1}||_{L^{\infty}} ||\Lambda^{s-2}\nabla \widetilde{\theta}_n||_{L^2}+||\Lambda^{s-1}\mathbf{u}_{n-1}||_{L^6}||\nabla \widetilde{\theta}_n||_{L^{3}}\\
&\lesssim ||\Lambda^{5/2^{+}}\mathbf{u}_{n-1}||_{L^{2}}||\Lambda^{s-1} \widetilde{\theta}_n||_{L^2}+||\Lambda^{s}\mathbf{u}_{n-1}||_{L^2}||\Lambda^{3/2^{+}} \widetilde{\theta}_n||_{L^{2}}\\
&\lesssim ||\Lambda^{s}\mathbf{u}_{n-1}||_{L^{2}}||\Lambda^{s-1} \widetilde{\theta}_n||_{L^2}.\\[10 pt]
\bullet \quad ||\Lambda^{s-1}(\widetilde{\mathbf{u}}_{n-1}\cdot\nabla\theta_{n-1})||_{L^2}&\lesssim  ||\widetilde{\mathbf{u}}_{n-1}||_{L^{\infty}} ||\Lambda^{s-1}\nabla \theta_{n-1}||_{L^2}+||\Lambda^{s-1}\widetilde{\mathbf{u}}_{n-1}||_{L^2}||\nabla \theta_{n-1}||_{L^{\infty}}\\
&\lesssim ||\Lambda^{3/2^{+}}\widetilde{\mathbf{u}}_{n-1}||_{L^2} ||\Lambda^s\theta_{n-1}||_{L^2}+||\Lambda^{s-1}\widetilde{\mathbf{u}}_{n-1}||_{L^{2}}||\Lambda^{5/2^{+}}\theta_{n-1}||_{L^{2}}\\
&\lesssim ||\Lambda^{s-1}\widetilde{\mathbf{u}}_{n-1}||_{L^{2}} ||\Lambda^s\theta_{n-1}||_{L^2}.
\end{align*}
Combining these two inequalities with Lemma \ref{M_order0} and Corollary \ref{M3_order_zero}  we have proved that there exists two positive constants satisfying $\mathfrak{m}_{\star}\leq \widehat{M_3}(\mathbf{k})\leq \mathfrak{m}^{\star}$ for all $k\in\Z^3_{\star}$ such that:
\begin{align*}
||\left[\mathbf{u}_{n-1}\cdot\nabla,\Lambda^{s-1}\right]\widetilde{\theta}_n||_{L^2}+||\Lambda^{s-1}(\widetilde{\mathbf{u}}_{n-1}\cdot\nabla\theta_{n-1})||_{L^2}&\lesssim \mathfrak{m}^{\star} ||\Lambda^s \theta_{n-1}||_{L^2} \left(||\Lambda^{s-1}\widetilde{\theta}_n||_{L^2}+||\Lambda^{s-1} \widetilde{\theta}_{n-1}||_{L^2} \right).
\end{align*}
Hence, combining this estimate with (\ref{lwp_parte_A}), we arrive to:
\begin{align*}
\tfrac{1}{2}\partial_t||\Lambda^{s-1}\widetilde{\theta}_n||_{L^2}^2+&||\sqrt{M_3}\,[\Lambda^{s-1}\widetilde{\theta}_n]||_{L^2}^2 \leq C_s  \mathfrak{m}^{\star} ||\Lambda^{s-1}\widetilde{\theta}_n||_{L^2} ||\Lambda^s \theta_{n-1}||_{L^2} \left(||\Lambda^{s-1}\widetilde{\theta}_n||_{L^2}+||\Lambda^{s-1} \widetilde{\theta}_{n-1}||_{L^2} \right)\\
&\leq \frac{C_s \mathfrak{m}^{\star}}{\sqrt{\mathfrak{m}_{\star}}}||\sqrt{M_3}\,[\Lambda^{s-1}\widetilde{\theta}_n]||_{L^2} ||\Lambda^s \theta_{n-1}||_{L^2} \left(||\Lambda^{s-1}\widetilde{\theta}_n||_{L^2}+||\Lambda^{s-1} \widetilde{\theta}_{n-1}||_{L^2} \right)\\
&\leq ||\sqrt{M_3}\,[\Lambda^{s-1}\widetilde{\theta}_n]||_{L^2}^2+\frac{1}{2}\left(\frac{C_s \mathfrak{m}^{\star}}{\sqrt{\mathfrak{m}_{\star}}}\right)^2||\Lambda^s \theta_{n-1}||_{L^2}^2\left(||\Lambda^{s-1}\widetilde{\theta}_n||_{L^2}^2+||\Lambda^{s-1}\widetilde{\theta}_{n-1}||_{L^2}^2\right)
\end{align*}
and, as consequence of (\ref{induction_assumption}) we get:
\begin{align*}
\partial_t||\Lambda^{s-1}\widetilde{\theta}_n||_{L^2}^2(t)&\leq \left(\frac{C_s \mathfrak{m}^{\star}}{\sqrt{\mathfrak{m}_{\star}}}\right)^2||\Lambda^s \theta_{n-1}||_{L^2}^2(t)\left(||\Lambda^{s-1}\widetilde{\theta}_n||_{L^2}^2(t)+||\Lambda^{s-1}\widetilde{\theta}_{n-1}||_{L^2}^2(t)\right)\\
&\leq 2\left(\frac{C_s \mathfrak{m}^{\star}}{\sqrt{\mathfrak{m}_{\star}}}\right)^2||\Lambda^s \theta_{0}||_{L^2}^2\left(||\Lambda^{s-1}\widetilde{\theta}_n||_{L^2}^2(t)+||\Lambda^{s-1}\widetilde{\theta}_{n-1}||_{L^2}^2(t)\right).
\end{align*}
Hence, applying Gr\"onwall's inequality (see \cite[p.~624]{Evans}) we obtain that:
\begin{align*}
||\Lambda^{s-1}\widetilde{\theta}_n||_{L^2}^2(t)&\leq \exp{\left[2\left(\frac{C_s\,\mathfrak{m}^{\star}}{\sqrt{\mathfrak{m}_{\star}}}  ||\Lambda^s\theta_{0}||_{L^2}\right)^2\,t\right]} ||\Lambda^{s-1}\widetilde{\theta}_n||_{L^2}^2(0)\\
&+\exp{\left[2\left(\frac{C_s\,\mathfrak{m}^{\star}}{\sqrt{\mathfrak{m}_{\star}}}  ||\Lambda^s\theta_{0}||_{L^2}\right)^2\,t\right]} 2\left(\frac{C_s\,\mathfrak{m}^{\star}}{\sqrt{\mathfrak{m}_{\star}}}  ||\Lambda^s\theta_{0}||_{L^2}\right)^2\int_{0}^{t}||\Lambda^{s-1}\widetilde{\theta}_{n-1}||_{L^2}^2(\tau)\,d\tau
\end{align*}
and as by definition $\widetilde{\theta}_n(\mathbf{x},0)=0$, we have for $0\leq t\leq T$ that:
$$||\Lambda^{s-1}\widetilde{\theta}_n||_{L^2}^2(t)\leq ||\Lambda^{s-1}\widetilde{\theta}_{n-1}||_{L^{\infty}(0,T;L^2)}^2 \exp{\left[2\left(\frac{C_s\,\mathfrak{m}^{\star}}{\sqrt{\mathfrak{m}_{\star}}}  ||\Lambda^s\theta_{0}||_{L^2}\right)^2\,T\right]} 2\left(\frac{C_s\,\mathfrak{m}^{\star}}{\sqrt{\mathfrak{m}_{\star}}}  ||\Lambda^s\theta_{0}||_{L^2}\right)^2\,T.$$
Here, after recalling (\ref{value_T_1}), if we let $T$ be such that:
$$T:=\frac{\min\{\log 2,1/6\}}{\left(\frac{C_s\,\mathfrak{m}^{\star}}{\sqrt{\mathfrak{m}_{\star}}} ||\Lambda^s\theta_0||_{L^2}\right)^2}$$
we obtain that:
$$||\Lambda^{s-1}\widetilde{\theta}_n||_{L^{\infty}(0,T;L^2)}^2\leq \frac{1}{2}||\Lambda^{s-1}\widetilde{\theta}_{n-1}||_{L^{\infty}(0,T;L^2)}^2.$$
Consequently $\{\theta_n\}_{n\geq 1}$ is Cauchy in $L^{\infty}(0,T;H^{s-1})$ and hence $\theta_n$ converges strongly to $\theta$ in $L^{\infty}(0,T;H^{s-1})$. Nothing that $s-1>3/2$, this shows that the strong convergence occurs in a H\"older space, which is sufficient to prove that the limiting function $\theta\in L^{\infty}(0,T;H^s)$ is a solution of the initial value problem (\ref{problema_perturbado}).

To conclude the proof of the theorem we prove the uniqueness.  We note that if $\theta^{(1)}$ and $\theta^{(2)}$ are two solutions of (\ref{problema_perturbado}), then $\theta^{\sharp}=\theta^{(1)}-\theta^{(2)}$ solves
\begin{equation}\label{problema_unicidad}
\left\{
\begin{array}{rl}
\partial_t\theta^{\sharp}(\mathbf{x},t)+\mathbf{u}^{(1)}(\mathbf{x},t)\cdot\nabla\theta^{\sharp}(\mathbf{x},t)+\mathbf{u}^{\sharp}(\mathbf{x},t)\cdot\nabla\theta^{(2)}(\mathbf{x},t)&=-M_3[\theta^{\sharp}](\mathbf{x},t) \\
\mathbf{u}^{\sharp}(\mathbf{x},t)&=\phantom{-}M[\theta^{\sharp}](\mathbf{x},t)\\
\theta^{\sharp}(\mathbf{x},0)&=\phantom{-}0.
\end{array}
\right.
\end{equation}
An $L^2$ estimate on (\ref{problema_unicidad}) shows that $\theta^{\sharp}(\mathbf{x},t)=0$ for all $t\in[0,T)$, since $\theta^{(2)}\in L^{\infty}(0,T;H^{5/2^{+}})$ and the frequency support of $\theta^{\sharp}$ belongs to $\mathrm{X}_{\mathfrak{C}}$, due to the fact that $\text{supp}(\widehat{\theta^{(i)}}(t))\subset \mathrm{X}_{\mathfrak{C}}$ for $t\in [0,T)$ and \nolinebreak $i=1,2.$
\end{proof}

\section{Global existence in $H^{5/2^{+}}(\T^3)$ for frequency-localized initial data}\label{Sec_4}
This section is devoted to prove the main result of this paper:

\begin{thm}\label{main_thm}
Fixed $\mathfrak{C}>0$ and the frequency straight line $\mathrm{X}_{\mathfrak{C}}$. Let $\theta_0\in H^s(\T^3)$ with zero vertical mean and frequency support in $\mathrm{X}_{\mathfrak{C}}$ such that $||\theta_0||_{H^{\kappa}}\leq \epsilon_0$ where $\epsilon_0$ is give by (\ref{epsilon_0_def}) and $\kappa:=\tfrac{1}{\alpha}+\tfrac{5}{2}^{+}$ for $\alpha\in(0,1)$. Then, the solution of the non-diffusive MG equation (\ref{active_scalar}) with  initial datum $\Theta(\mathbf{x},0)=\Omega(x_3)+\theta_0(\mathbf{x})$ exists globally in time and satisfies the following exponential decay to the steady state:
$$||\Theta-\Omega||_{H^s}(t)\equiv||\theta||_{H^s}(t)\lesssim ||\theta_0||_{H^s} \exp(-\mathfrak{m}_{\star}t).$$
\end{thm}
In the next sections we give the proof of this result.

\subsection{Energy methods for the MG equation}\label{Sec_4.1}
For $s>5/2$ and initial data $\theta_0\in H^s(\T^3)$ with zero vertical mean and $\text{supp}(\widehat{\theta_0})\subset \mathrm{X}_{\mathfrak{C}}$, there exists $T>0$ such that $\theta(t)\in H^s(\T^3)$ and  $\text{supp}(\widehat{\theta}(t))\subset \mathrm{X}_{\mathfrak{C}}$ for all $t\in[0,T)$.

\subsubsection{A priori energy estimates}
In what follows, we assume that $\theta(t)\in H^s(\T^3)$ is a solution of (\ref{problema_perturbado}) and  the frequency support $\text{supp}(\widehat{\theta}(t))\subset \mathrm{X}_{\mathfrak{C}}$ for any $t\geq 0$. Then, the following estimate holds:

\begin{equation*}
\partial_t||\theta||_{H^s}^2(t)\lesssim -\left[1-C\,||\theta||_{H^{5/2^{+}}}(t)\right]||\theta||_{H^s}^2(t).
\end{equation*}
First of all, we will perform the basic $H^s$-energy estimate for
\begin{equation}\label{Evolucion_theta}
\partial_t\theta(\mathbf{x},t) +\mathbf{u}(\mathbf{x},t)\cdot\nabla\theta(\mathbf{x},t) = -M_3[\theta](\mathbf{x},t)  \\
\end{equation}
where $\mathbf{u}(\mathbf{x},t)=\mathbf{M}[\theta](\mathbf{x},t)$ and the initial data $\theta_0(\mathbf{x})$ has zero vertical mean and frequency support in $\mathrm{X}_{\mathfrak{C}}$.\vspace{0.2 cm}

\noindent
\textbf{$L^2$-estimate:} We multiply (\ref{Evolucion_theta}) by $\theta$ and integrate over $\mathbb{T}^3$. Then:
$$\tfrac{1}{2}\partial_t ||\theta||_{L^2}^2=-\int_{\mathbb{T}^3} \theta\, M_3[\theta]\,d\mathbf{x}-\int_{\mathbb{T}^3}\theta\, \left(\mathbf{u}\cdot\nabla\right)\theta\,d\mathbf{x}.$$
Therefore, using Plancherel's theorem and (\ref{square_root_M3}), we obtain that:
\begin{equation}\label{L2_norm}
\tfrac{1}{2}\partial_t ||\theta||_{L^2}^2=-\sum_{\mathbf{k}\in \Z_{\star}^{3}}  \widehat{M_3}(\mathbf{k})\,|\hat{\theta}(\mathbf{k})|^2=-||\sqrt{M_3}\,[\theta]||_{L^2}^2.
\end{equation}

\noindent
\textbf{$\dot{H}^s$-estimate:}
Applying $\Lambda^s$ to (\ref{Evolucion_theta}) and taking an $L^2$ inner product with $\Lambda^s\theta$ we obtain:
\begin{align*}
\tfrac{1}{2}\partial_t ||\theta||_{\dot{H}^s}^2&=-\int_{\mathbb{T}^3}\Lambda^s\theta\, M_3[\Lambda^s\theta]\,d\mathbf{x}-\int_{\mathbb{T}^3}\Lambda^s\theta\,\Lambda^s[(\mathbf{u}\cdot\nabla)\theta]\, d\mathbf{x}=I_1+I_2.
\end{align*}
First of all we study $I_1$. As before, by Plancherel's theorem and the square roor of $M_3$ given by (\ref{square_root_M3}) we get:
\begin{equation}\label{I1}
I_1=-\sum_{\mathbf{k}\in \Z_{\star}^{3}}  \widehat{M_3}(\mathbf{k})\,|\widehat{\Lambda^s\theta}(\mathbf{k})|^2=-||\sqrt{M_3}\,[\Lambda^s\theta]||_{L^2}^2=-||\sqrt{M_3}\,[\theta]||_{\dot{H}^s}^2.
\end{equation}
%
Secondly, we study $I_2$. Below, we use the fact that $\nabla\cdot\mathbf{u}=0$ in order to obtain a commutator operator:
\begin{align*}
I_2&=-\int_{\mathbb{T}^3}\Lambda^s\theta\,\Lambda^s[(\mathbf{u}\cdot\nabla)\theta]\, d\mathbf{x}\,\, \pm\,\,\int_{\mathbb{T}^3}\Lambda^s\theta\,(\mathbf{u}\cdot\nabla)\Lambda^s\theta\, d\mathbf{x}=-\int_{\mathbb{T}^3}\Lambda^s\theta\,\left[\Lambda^s,\mathbf{u}\cdot\nabla\right]\theta\,d\mathbf{x}.
\end{align*}
Hence, using the commutator estimate (\ref{eq:prop:commutator}) and the Sobolev embedding $L^{\infty}(\T^3)\hookrightarrow H^{3/2^{+}}(\T^3)$ we arrive to:
\begin{align*}
I_2&\leq ||\Lambda^s\theta||_{L^2}\,||\left[\Lambda^s,\mathbf{u}\cdot\nabla\right]\theta||_{L^2}\lesssim  ||\Lambda^s\theta||_{L^2}\left(||\nabla \mathbf{u}||_{L^{\infty}}\,||\Lambda^{s-1}\nabla\theta||_{L^2} +||\Lambda^s\mathbf{u}||_{L^2}\,||\nabla\theta||_{L^{\infty}}\right)\\
&\lesssim ||\theta||_{\dot{H}^s}\left(||\mathbf{u}||_{H^{5/2^{+}}}\,||\theta||_{\dot{H}^s} +||\mathbf{u}||_{\dot{H}^s}\,||\theta||_{H^{5/2^{+}}}  \right).
\end{align*}
Since $\text{supp}(\widehat{\theta}(t))\subset \mathrm{X}_{\mathfrak{C}}$ as lons as the solution exists, applying Lemma \ref{M_order0} we have that the Fourier operator $\widehat{\textbf{M}}(\mathbf{k})$ restricted to $\mathbf{k}\in\mathrm{X}_{\mathfrak{C}}$ behaves like a zero order operator. In particular, for $s>5/2$ we have that:
\begin{equation}\label{I2}
I_2\leq C_s \mathfrak{m}^{\star} ||\theta||_{\dot{H}^s}^2\,||\theta||_{H^{5/2^{+}}} 
\end{equation}
where $\mathfrak{m}^{\star}(\mathfrak{C})$ blows-up as $\mathfrak{C}$ tends to infinity.
Putting together (\ref{I1}) and (\ref{I2}), for $s>5/2$ we have that:
\begin{equation}\label{Hs-norm}
\tfrac{1}{2}\partial_t ||\theta||_{\dot{H}^s}^2\leq C_s \mathfrak{m}^{\star} ||\theta||_{\dot{H}^s}^2\,||\theta||_{H^{5/2^{+}}}-||\sqrt{M_3}\,[\theta]||_{\dot{H}^s}^2.
\end{equation}

To sum up, we have proved the next energy estimate.

\begin{thm}\label{energy_estimate} 
Let $\theta(t)\in H^s(\T^3)$ be a solution of (\ref{problema_perturbado}) with zero mean and $\text{supp}(\hat{\theta}(t))\subset \mathrm{X}_{\mathfrak{C}}$ for any $t\geq 0$. Then, for $s>5/2$ the following estimate holds:
\begin{equation*}
\tfrac{1}{2}\partial_t ||\theta||_{H^s}^2(t)\leq -\mathfrak{m}_{\star}\left[1-\left(\tfrac{C_s \mathfrak{m}^{\star}}{\mathfrak{m}_{\star}}\right)||\theta||_{H^{5/2^{+}}}(t)\right] ||\theta||_{H^s}^2(t).
\end{equation*}
\end{thm}
\begin{proof}
Putting together (\ref{L2_norm}) with (\ref{Hs-norm}) we arrive to: 
$$\tfrac{1}{2}\partial_t ||\theta||_{H^s}^2(t)\leq C_s \mathfrak{m}^{\star} ||\theta||_{H^s}^2(t)\,||\theta||_{H^{5/2^{+}}}(t)-||\sqrt{M_3}\,[\theta]||_{H^s}^2(t).$$
Using that $\text{supp}(\hat{\theta}(t))\subset \mathrm{X}_{\mathfrak{C}}$ for any $t\geq 0$ and Lemma \ref{M3_order_zero}, we obtain that:
$$\tfrac{1}{2}\partial_t ||\theta||_{H^s}^2(t)\leq C_s \mathfrak{m}^{\star} ||\theta||_{H^s}^2(t)\,||\theta||_{H^{5/2^{+}}}(t)-\mathfrak{m}_{\star}||\theta||_{H^s}^2(t)$$
where $\mathfrak{m}_{\star}(\mathfrak{C})$ goes to zero as $\mathfrak{C}$ tends to infinity.  Rewriting it, we have achieved our goal.
%
\end{proof}
So, as consequence, we establish a ``small'' data global existence result.

\begin{cor}
Fixed $s>5/2$. Let $\theta_0\in H^s(\T^3)$ with zero vertical mean and $\text{supp}(\widehat{\theta}_0)\subset\mathrm{X}_{\mathfrak{C}}$ such that $||\theta_0||_{H^{5/2^{+}}}\leq \varepsilon$ small enough. Then, the solution exists globally in time and satisfies a maximum principle:
$$||\theta||_{H^s}(t)\leq ||\theta_0||_{H^s}.$$
\end{cor}
%

In the following section we will improve the previous result. Using a perturbative argument, we are able to derive explicit expressions that quantify the decay rates.  This leads to an asymptotic stability result of the steady state.

\subsection{Linear \& non-linear estimates}\label{Sec_4.2}
The linearized equation gives very good decay properties. Hence, the main achievement of this section is to control the nonlinearity, so that it would not destroy the decay provided by the linearized equation.

\subsubsection{Linear decay}
We approach the question of global well-posedness for a small initial data from a perturbative point of
view, i.e., we see (\ref{problema_perturbado}) as a non-linear perturbation of the linear problem. The linearized equation around the trivial solution $(\theta,\mathbf{u})=(0,0)$ reads as
\begin{equation}\label{problema_lineal}
\left\{
\begin{array}{rl}
\partial_t\theta(\mathbf{x},t)  +M_3[\theta](\mathbf{x},t)&= \phantom{+} 0  \\
\theta(\mathbf{x},0)&=\phantom{+} \theta_0(\mathbf{x})
\end{array}
\right.
\end{equation}
where the initial data $\theta_0\in H^s(\T^3)$ has zero vertical mean and frequency support in $\mathrm{X}_{\mathfrak{C}}$.\vspace{0.2 cm}

As $M_3$ is a positive operator, we derive the exponential decay in time of solutions to the linear problem with decay rate depending of the frequency support of the initial data.

\begin{cor}
The solution of (\ref{problema_lineal}) with initial data  $\theta_0\in H^s(\T^3)$ and with zero vertical mean and frequency support in $\mathrm{X}_{\mathfrak{C}}$ satisfies that
$$||\theta||_{H^s}(t)\leq ||\theta_0||_{H^s}\exp(-\mathfrak{m}_{\star}t),$$
where $\mathfrak{m}_{\star}(\mathfrak{C})$ goes to zero as $\mathfrak{C}$ tends to infinity.
\end{cor}

\subsubsection{Non-linear decay}
Next, we will show how this decay of the linear solutions can be used to
establish the stability of the stationary solution $(\theta, \mathbf{u})=(0,0)$ for the general problem (\ref{problema_perturbado}). When perturbing
around it, we get the following system:
\begin{equation}\label{perturbado_lineal}
\left\{
\begin{array}{rl}
\partial_t\theta(\mathbf{x},t) +M_3[\theta](\mathbf{x},t) &= -\mathbf{u}(\mathbf{x},t)\cdot\nabla\theta(\mathbf{x},t)  \\
\theta(\mathbf{x},0)&=\phantom{-}\theta_0(\mathbf{x})
\end{array}
\right.
\end{equation}
where $\mathbf{u}(\mathbf{x},t)=\textbf{M}[\theta](\mathbf{x},t)$ and the initial data $\theta_0$ satisfies the same hypothesis.
Using Duhamel's formula, we write the solution of (\ref{perturbado_lineal}) as:
$$\theta(\mathbf{x},t)=e^{\mathcal{L}(t)}\theta(\mathbf{x},0)-\int_{0}^{t}e^{\mathcal{L}(t-\tau)}\left[\mathbf{u}\cdot\nabla\theta\right](\mathbf{x},\tau)\,d\tau$$
where $\mathcal{L}(t)$ denotes the solution operator of the associated linear problem (\ref{problema_lineal}). Therefore, we have that:
\begin{equation}\label{Duhamel's_formula}
||\theta||_{H^s}(t)\leq ||\theta_0||_{H^s}\exp(-\mathfrak{m}_{\star}t)+\int_{0}^{t}||\mathbf{u}\cdot\nabla\theta||_{H^s}(\tau)\exp\left(-\mathfrak{m}_{\star}(t-\tau)\right)\,d\tau.
\end{equation}

\subsection{The bootstraping}\label{Sec_4.3}
We now demonstrate the bootstrap argument used to prove our goal. The general
approach here is a typical continuity argument that has been used successfully in a plethora of other cases.
Theorem \ref{energy_estimate} tell us that the following estimate holds for $s>5/2$:
\begin{equation}\label{energy_estimate_boostraping}
\tfrac{1}{2}\partial_t ||\theta||_{H^s}^2(t)\leq -\mathfrak{m}_{\star}\left[1-\left(\tfrac{C_s \mathfrak{m}^{\star}}{\mathfrak{m}_{\star}}\right)||\theta||_{H^{5/2^{+}}}(t)\right] ||\theta||_{H^s}^2(t).
\end{equation}
In the following, let $\alpha\in(0,1)$ be a free parameter  and $\kappa:=\frac{1}{\alpha}+\frac{5}{2}^{+}$. We want to prove that $||\theta||_{H^{5/2^{+}}}$ decays in time. This will allow us to close the energy estimate and \hyphenation{fi-nish} finish the proof. We will prove it through a bootstrap argument, where the main ingredient is the estimate $(\ref{energy_estimate_boostraping})$.\vspace{0.2 cm}


\subsubsection{Exponential decay of $||\theta||_{H^{5/2^{+}}}$}
In order to control $||\theta||_{H^{5/2^{+}}}(t)$ in time, we have the following result.
\begin{lemma}\label{decay_h^5/2}
Assume that $||\theta_0||_{H^{\kappa}}\leq \epsilon$ and $||\theta||_{H^{\kappa}}(t)\leq 4\epsilon$  for all $t\in[0,T]$ where $\kappa=\frac{1}{\alpha}+\frac{5}{2}^{+}$ with $0<\alpha<1$. Then, we have:
$$||\theta||_{H^{5/2^{+}}}(t)\leq 2\epsilon\exp(-\mathfrak{m}_{\star}\,t) \qquad \text{for all}\quad  t\in[0,T].$$
\end{lemma}
\begin{proof}
Duhamel's formula (\ref{Duhamel's_formula}) give us:
$$||\theta||_{H^{5/2^{+}}}(t)\leq ||\theta_0||_{H^{5/2^{+}}}\exp(-\mathfrak{m}_{\star}\,t)+\int_{0}^{t}||\mathbf{u}\cdot\nabla\theta||_{H^{5/2^{+}}}(\tau)\exp\left(-\mathfrak{m}_{\star}\,(t-\tau)\right)\,d\tau$$
and using the algebraic properties of Sobolev spaces we have that:
$$||\mathbf{u}\cdot\nabla\theta||_{H^{5/2^{+}}}\lesssim ||\mathbf{u}||_{H^{5/2^{+}}}||\theta||_{H^{7/2^{+}}}
\lesssim \mathfrak{m}^{\star} ||\theta||_{H^{5/2^{+}}}||\theta||_{H^{7/2^{+}}}.$$
The last inequality is due to Lemma \ref{M_order0} and the fact that $\text{supp}(\widehat{\theta}(t))\subset \mathrm{X}_{\mathfrak{C}}$ as long as the solution exists. Moreover, due to the well-known Gagliardo-Nirenberg interpolation inequality:
$$||\theta||_{H^{7/2^{+}}}\lesssim ||\theta||_{H^{5/2^{+}}}^{1-\alpha} ||\theta||_{H^{1/\alpha+5/2^{+}}}^{\alpha} \qquad \text{with}\quad0<\alpha<1$$
we arrive to
$$||\theta||_{H^{5/2^{+}}}(t)\leq ||\theta_0||_{H^{5/2^{+}}}\exp(-\mathfrak{m}_{\star}\,t)+\int_{0}^{t} C_{\alpha} \,\mathfrak{m}^{\star} ||\theta||_{H^{5/2^{+}}}^{2-\alpha}(\tau)||\theta||_{H^{1/\alpha+5/2^{+}}}^{\alpha}(\tau)\exp\left(-\mathfrak{m}_{\star}\,(t-\tau)\right)\,d\tau.$$

By hypothesis, we have that $||\theta||_{H^{1/\alpha+5/2^{+}}}(t)\leq  4\epsilon$ on the interval $[0,T]$. Then, we obtain that:
\begin{equation}\label{step2_boostraping}
||\theta||_{H^{5/2^{+}}}(t)\leq \epsilon\exp(-\mathfrak{m}_{\star}\,t)+\int_{0}^{t} (4\epsilon)^{\alpha} C_{\alpha} \,\mathfrak{m}^{\star} ||\theta||_{H^{5/2^{+}}}^{2-\alpha}(\tau)\exp\left(-\mathfrak{m}_{\star}\,(t-\tau)\right)\,d\tau.
\end{equation}
In particular, there exist $0<T^{\star}(\alpha)\leq T$ such that for $t\in[0,T^{\star}]$ we have that:
\begin{equation}\label{step1_boostraping}
||\theta||_{H^{5/2^{+}}}(t)\leq 4\epsilon\exp(-\mathfrak{m}_{\star}\,t)
\end{equation}
If we restrict to $0\leq t\leq T^{\star}$ and we apply (\ref{step1_boostraping}) into (\ref{step2_boostraping}), we have:
\begin{align*}
||\theta||_{H^{5/2^{+}}}(t)&\leq \epsilon\exp(-\mathfrak{m}_{\star}\,t) +(4\epsilon)^2 C_{\alpha} \mathfrak{m}^{\star}\exp(-\mathfrak{m}_{\star}\,t)\int_{0}^{t}\exp(-(1-\alpha)\mathfrak{m}_{\star}\tau)\,d\tau\\
&\leq \epsilon\exp(-\mathfrak{m}_{\star}\,t)\left[1+\epsilon \frac{4^2 C_{\alpha} \mathfrak{m}^{\star}}{(1-\alpha)\mathfrak{m}_{\star}} \right].
\end{align*}
Taking $0<\epsilon<\frac{(1-\alpha)\mathfrak{m}_{\star}}{4^2 C_{\alpha} \mathfrak{m}^{\star}}$ we have proved that:
$$||\theta||_{H^{5/2^{+}}}(t)\leq 2\epsilon\exp(-\mathfrak{m}_{\star}\,t)$$
for all $t\in[0,T^{\star}]$ and, by continuity, for all $t\in[0,T]$.
\end{proof}

\subsubsection{A new boostraping argument.}
In order to control $||\theta||_{H^{\kappa}}(t)$ in time, we have the following result.
\begin{lemma}\label{smallness_H^kappa}
Assume that $||\theta_0||_{H^{\kappa}}\leq \epsilon$ and $||\theta||_{H^{\kappa}}(t)\leq 4\epsilon$  for all $t\in[0,T]$ where $\kappa=\frac{1}{\alpha}+\frac{5}{2}^{+}$ with $0<\alpha<1$. Then, we have that:
$$||\theta||_{H^{\kappa}}(t)\leq 2\epsilon \qquad \text{for all}\quad  t\in[0,T].$$
\end{lemma}
\begin{proof}
Applying Gr\"onwall's inequality into (\ref{energy_estimate_boostraping}) and  Lemma \ref{decay_h^5/2}, for $t\in[0,T]$ we have that:
\begin{align*}
||\theta||_{H^{\kappa}}(t)&\leq ||\theta_0||_{H^{\kappa}}\,\exp\left[-\mathfrak{m}^{\star}\int_{0}^{t}\left(1-\left(\tfrac{C_{\kappa} \mathfrak{m}^{\star}}{\mathfrak{m}^{\star}}\right)||\theta||_{H^{5/2^{+}}}(\tau)\right)\,d\tau \right]\\
&\leq ||\theta_0||_{H^{\kappa}}\exp\left( \frac{2\epsilon\, C_{\kappa} \mathfrak{m}^{\star}}{\mathfrak{m}_{\star}}\right).
\end{align*}
Taking $0<\epsilon< \frac{\log \sqrt{2}}{ C_{\kappa} }\frac{\mathfrak{m}_{\star}}{\mathfrak{m}^{\star}}$ we have proved that $||\theta||_{H^{\kappa}}(t) \leq 2\epsilon$ for all $t\in[0,T].$
\end{proof}
Therefore, it is natural to define a ``smallness'' parameter $\epsilon_0$ given by:
\begin{equation}\label{epsilon_0_def}
\epsilon_0:=\min\left\lbrace \frac{(1-\alpha)}{4^2 C_{\alpha}},\frac{\log \sqrt{2}}{ C_{\kappa} }\right\rbrace \frac{\mathfrak{m}_{\star}}{\mathfrak{m}^{\star}}.
\end{equation}
In consequence, a straightforward  combination of Lemma \ref{decay_h^5/2} and Lemma \ref{smallness_H^kappa} give us:
\begin{cor}\label{cor_final}
Let $\theta_0\in H^{\kappa}(\T^3)$ such that $||\theta_0||_{H^{\kappa}}\leq \epsilon$ with $0<\epsilon\leq \epsilon_0$. Then, for all $t\geq 0$ we have that:
$$||\theta||_{H^{\kappa}}(t) \leq 2\epsilon \qquad \text{and} \qquad ||\theta||_{H^{5/2^{+}}}(t)\leq 2\epsilon\exp(-\mathfrak{m}_{\star}\,t).$$
\end{cor}

\subsubsection{Exponential decay of $||\theta||_{H^s}$ with $s>\tfrac{7}{2}$}
We have proved the exponential decay in time of $||\theta||_{H^{5/2^{+}}}(t)$. Then, we are in the position to show how the bootstrap can be closed. This is merely a matter of collecting
the conditions established above and showing that they can indeed be satisfied.

\begin{lemma}\label{bootstrap_lemma}
Let $\theta_0\in H^s(\T^3)$ with $s\geq \kappa$ such that $||\theta_0||_{H^{\kappa}}\leq \epsilon$ where $0<\epsilon\leq \epsilon_0$. Then, for all $t\geq 0$ we have that:
$$||\theta||_{H^s}(t)\lesssim ||\theta_0||_{H^s} \exp(-\mathfrak{m}^{\star}t).$$
\end{lemma} 
\begin{proof}
Applying Gr\"onwall's inequality into (\ref{energy_estimate_boostraping}) we have:
$$||\theta||_{H^s}(t)\leq ||\theta_0||_{H^s}\,\exp\left[-\mathfrak{m}_{\star}\int_{0}^{t}\left(1-\left(\tfrac{C_s \mathfrak{m}^{\star}}{\mathfrak{m}_{\star}}\right)||\theta||_{H^{5/2^{+}}}(\tau)\right)\,d\tau \right].$$
The exponential decay of $||\theta||_{H^{5/2^{+}}}$ proved in Corollary \ref{cor_final} give us:
$$||\theta||_{H^s}(t)\leq ||\theta_0||_{H^s} \exp(-\mathfrak{m}_{\star}t) \exp\left(  2\epsilon\,C_s \frac{\mathfrak{m}^{\star}}{\mathfrak{m}_{\star}}\right)$$
and as $0<\epsilon\leq \epsilon_0$ there exists a constant $C=C(\alpha,s,\kappa)$ such that $||\theta||_{H^s}(t)\leq C ||\theta_0||_{H^s} \exp(-\mathfrak{m}_{\star}t).$
\end{proof}

%
\noindent
\textbf{Funding:} The author is  partially supported by Spanish National Research Project MTM2017-89976-P and ICMAT Severo Ochoa projects SEV-2011-0087 and SEV-2015-556.\vspace{0.2 cm}

\noindent
\textbf{Acknowledgements:}
The author thanks \'Angel Castro and Diego C\'ordoba for their valuable comments. 
The author acknowledges helpful conversations with Susan Friedlander and Roman Shvydkoy.

\bibliography{bibliografia}
\bibliographystyle{plain}

\Addresses

\end{document}